\documentclass{article}
\usepackage{graphicx}
\usepackage{setspace}
\usepackage{amsmath}
\usepackage{amssymb,amsfonts}
\usepackage{amsthm}
\usepackage[pdftex]{hyperref}
\newcommand{\Y}{\mathbb{Y}}
\newcommand{\la}{\lambda}
\newcommand{\Prob}{\operatorname{Prob}}

\newcommand{\Z}{\mathbb{Z}}

\newtheorem{theorem}{Theorem}[section]
\newtheorem{proposition}[theorem]{Proposition}
\newtheorem{lemma}[theorem]{Lemma}
\newtheorem{corollary}[theorem]{Corollary}

\newtheorem{definition}[theorem]{Definition}
\usepackage[all, knot]{xy}
\xyoption{arc}
\begin{document}
\title{Asymptotics of Plancherel measures for the infinite-dimensional unitary group}
\author{Alexei Borodin \and Jeffrey Kuan}
\maketitle
\begin{abstract}
We study a two-dimensional family of probability measures on
infinite Gelfand-Tsetlin schemes induced by a distinguished family
of extreme characters of the infinite-dimensional unitary group.
These measures are unitary group analogs of the well-known
Plancherel measures for symmetric groups.

We show that any measure from our family defines a determinantal
point process on $\Z_+\times\Z$, and we prove that in appropriate
scaling limits, such processes converge to two different extensions
of the discrete sine process as well as to the extended Airy and
Pearcey processes.
\end{abstract}

\section{Introduction}
Let $S(n)$ be the symmetric group of degree $n$. Denote by $\Y_n$
the set of partitions of $n$ or, equivalently, the set of Young
diagrams with $n$ boxes. It is well known that complex irreducible
representations of $S(n)$ are parameterized by elements of $\Y_n$;
we denote by $\dim\la$ the dimension of the irreducible
representation corresponding to $\la$. The probability distribution
$$
\Prob\{\lambda\}=\frac{\dim^2\la}{n!},\qquad \lambda\in \Y_n,
$$
on $\Y_n$ is called the {\it Plancherel measure\/} for $S(n)$. The
Plancherel weight of $\lambda\in \Y_n$ is the relative dimension of
the isotypic component of the regular representation of $S(n)$,
which transforms according to the irreducible representation
corresponding to $\la$. Hence, one has the following equality of
functions on $S(n)$:
$$
\delta_e=\sum_{\la\in \Y_n}
\frac{\dim^2\la}{n!}\,\frac{\chi^\la}{\dim\lambda}\,,
$$
where $\delta_e$ is the delta-function at the unity, and $\chi^\la$
is the irreducible character corresponding to $\la$.

Let $S(\infty)=\cup_{n\ge 1} S(n)$ be the group of finite
permutations of a countable set known as the {\it infinite symmetric
group}, see e.g. \cite{kn:K}. The group $S(\infty)$ has a rich
theory of characters (positive-definite central functions on the
group). For any character $\chi$ of $S(\infty)$ normalized by
$\chi(e)=1$, its restriction to the subgroup $S(n)$ of permutations
of first $n$ symbols is a convex combination of
$\{\chi^\lambda/\dim\la\}_{\la\in \Y_n}$. The coefficients
$\hat\chi_n(\lambda)$ form a probability measure on $\Y_n$; they are
a kind of Fourier transform of $\chi$.

There exists only one character $\chi$ of $S(\infty)$ for which the
rows and columns of the Young diagrams distributed according to
$\hat\chi_n$ grow sublinearly in $n$ as $n\to\infty$. This character
is the delta-function at the unity of $S(\infty)$, the corresponding
representation is the (bi)regular representation of $S(\infty)$ in
$\ell^2(S(\infty))$, and $\hat\chi_n$ is the Plancherel measure on
$\Y_n$ introduced above.

An analogous construction for the {\it infinite--dimensional unitary
group\/} $U(\infty)=\cup_{N\ge 1} U(N)$ yields a two-dimensional
family of characters of $U(\infty)$. Although the notion of regular
representation for $U(\infty)$ is meaningless, by comparing the
lists of the {\it extreme\/} (i.e., indecomposable) characters of
$S(\infty)$ and $U(\infty)$ one sees that the analog of $\delta_e$
on $S(\infty)$ is the family of characters
$$
\chi^{\gamma^+,\gamma^-}(U)=\exp\left
(\operatorname{Tr}\left(\gamma^+(U-1)+\gamma^-(U^{-1}-1)\right)\right),\qquad
U\in U(\infty),
$$
where $\gamma^\pm\ge 0$ are the parameters of the family. We will
provide details in Section 3, and for now let us just say that on
the level of Fourier transform, the set $\Y_n$ is replaced by the
set of $N$-tuples of integers $\lambda_1\ge\dots\ge\lambda_N$ which we
call {\it signatures\/} or {\it highest weights\/} of length $N$
(they parameterize irreducible representations of the unitary group
$U(N)$), and the corresponding probability distributions have the
form
$$
\gathered
P^{\gamma^+,\gamma^-}_N(\lambda_1,\dots,\lambda_N)=\operatorname{const}\cdot
\det\bigl[f^{(\gamma^+,\gamma^-)}_i(\lambda_j-j)\bigr]_{i,j=1}^N
\dim_{U(N)}(\la),\\
f_k^{(\gamma^+,\gamma^-)}(x)=\frac 1{2\pi i}\oint_{|z|=1}
\frac{e^{\gamma^+z+\gamma^-z^{-1}}dz}{z^{x+k+1}}\,,\qquad
k=1,2,\dots,
\endgathered
$$
where $\dim_{U(N)}(\la)$ is the dimension of the irreducible
representation of $U(N)$ with highest weight $\la$. We call the
measures $P_N^{\gamma^+,\gamma^-}$ {\it the Plancherel measures for
the infinite-dimensional unitary group\/}, and the present paper is
devoted to the study of these measures.

One source of interest to the Plancherel measures for symmetric
groups is the fact that the distribution of the largest part of
$\la\in \Y_n$ coincides with the distribution of the longest
increasing subsequence of uniformly distributed permutation in
$S(n)$. This fact can be restated in terms of a random growth model
in one space dimension called the polynuclear growth process (PNG).
Namely, the distribution of the height function for PNG with the
so-called droplet initial condition at any given point in space-time
coincides with the distribution of the largest part of $\la\in
\cup_{n\ge 0} \Y_n$ distributed according to the {\it Poissonized\/}
Plancherel measure
$$
\Prob\{\la\}=e^{-\theta^2}\left(\frac{\theta^{|\la|}\dim\la}{|\lambda|!}\right)^2,\qquad
\lambda\in\cup_{n\ge 0} \Y_n,
$$
where $|\lambda|$ is the number of boxes in the Young diagram
$\lambda$, and $\theta>0$ is a parameter, see
\cite{kn:PS}.

Quite similarly, the largest coordinate of a signature distributed
according to the Plancherel measure for $U(\infty)$ describes the
height function in another growth model in one space dimension
called PushASEP for the so-called step initial condition. This fact
can be established by direct comparison of Proposition 3.4 from \cite{kn:BF} and Theorem \ref{sectiontwotheorem} below.

The asymptotics of the Plancherel measure for $S(n)$ as $n\to\infty$
has been extensively studied. In the seventies, Logan and Shepp
\cite{kn:LS} and, independently, Vershik and Kerov
\cite{kn:VK}, \cite{kn:VK3}, discovered that Plancherel distributed
Young diagrams have a {\it limit shape\/}: In a suitable metric, the
measure on these Young diagrams scaled by $\sqrt{n}$ converges as
$n\to\infty$ to the delta-measure supported on a certain shape. In
the late nineties, more refined results were obtained. It was shown
that the random point process generated by the rows (or columns) of
the Plancherel distributed Young diagrams has two types of scaling
limits, in the ``bulk'' and at the ``edge'' of the limit shape. In
the limit, the former case yields the discrete sine determinantal
point process, while the latter case yields the Airy determinantal
point process, see \cite{kn:BDJ},
\cite{kn:BDJ2}, \cite{kn:OK2},
\cite{kn:BOO}, \cite{kn:J2}.

The main goal of the present paper is to prove similar asymptotics
results on scaling limits of random point processes related to more
complex measures $P_N^{\gamma^+,\gamma^-}$ with $N\to\infty$ and
$\gamma^\pm$ possibly dependent on $N$. Note that our results do not
imply the existence of the limit shape in any of the cases we
consider, although they strongly suggest that in some cases the
limit shape does exist, and they predict what it looks like. For a
discussion of the relationship between ``local'' results on point
processes and ``global'' measure concentration properties see Remark
1.7 of \cite{kn:BOO}, \S 1 of \cite{kn:BO4}.

Let us describe our results in more detail.

It is convenient to represent a signature
$\lambda=\{\lambda_1\ge\dots\ge \lambda_N\}$ as a pair of
partitions, one partition $\lambda^+$ consists of positive parts of
$\lambda$ while the other one $\lambda^-$ consists of absolute
values of negative parts of $\lambda$. When the parameters
$\gamma^\pm$ are independent of $N$, they describe (see Section
2) the asymptotic behavior of $|\lambda^{\pm}|$, namely
$|\lambda^{\pm}|\sim \gamma^\pm N$, as $N\to\infty$. This asymptotic
relation remains true in other situations as well, and it is helpful
to keep it in mind when going through the limit transitions below.

Our first result describes what happens when
$\gamma^\pm\sim N^{-1}$ as $N\to\infty$. Then one expects that
$|\lambda^+|$, $|\lambda^-|$ remain finite in the limit, and indeed
the measures $P_N^{(\gamma^+,\gamma^-)}$ converge to the product of
two independent copies of the Poissonized Plancherel measures for
the symmetric groups that live on $\lambda^\pm$.

The next possibility to consider is when $\gamma^\pm$ are
independent of $N$. The case when $\gamma^-=0$ was considered by
Kerov \cite{kn:K2}, who proved the existence of the limit shape
and showed that the limit shape coincides with that for the
Plancherel measures for symmetric groups. We show that when both
parameters $\gamma^\pm$ are fixed and nonzero, the random point
processes describing $\lambda^\pm$ asymptotically behave as though
$\lambda^\pm$ represent two independent copies of the Poissonized
Plancherel measures for the symmetric group with Poissonization
parameters $\gamma^\pm N\to\infty$.

The most interesting case is when $\gamma^\pm$ grow at the same rate
as $N$. Biane \cite{kn:Bi} proved that when $\gamma^-=0$, the
corresponding measure has a limit shape that depends on the limiting
value of the ratio $\gamma^+/N$. We consider the case when both
parameters are nonzero and investigate the asymptotic behavior of
the random point process that describes our random signatures.

Even though we do not prove the existence of the limit shape, it is
convenient to use the hypothetical limit shape inferred from the
limit of the density function to describe the results. There are
three possibilities: The limit shapes of $\lambda^\pm$ scaled by $N$
do not touch (that happens when $\gamma^\pm/N$ are small), when they
barely meet, and when they have already met, see Figure~\ref{limitcurves2}
in the body of the paper. Accordingly, there are three types of
local behavior one can expect: The bulk, the edge, where the limit
shape becomes tangent to one of the axes, and the point when the
edges of the limit shapes for $\lambda^\pm$ meet. We compute the
local scaling limits of the correlation functions for the random
point process describing our signatures, and obtain the correlation
functions of the discrete sine, Airy, and Pearcey determinantal
processes in the three cases above.

As a matter of fact, we consider probability measures on a more
general object than signatures. Every character of $U(\infty)$
naturally defines a probability measure on Gelfand-Tsetlin schemes
(a kind of infinite semistandard Young tableaux), see Section 2 and
references therein. The corresponding measures on signatures of
length $N$ are certain projections of the measure on Gelfand-Tsetlin
schemes. In particular, every character from our two-dimensional
Plancherel family yields a measure on Gelfand-Tsetlin schemes, and
that is what we study asymptotically. We interpret each scheme as a
point configuration in $\Z\times\Z_+$, and compute the scaling
limits of correlation functions of the arising two-dimensional
random point processes. The results are appropriate (determinantal)
time-dependent extensions of the limiting processes mentioned above.

The proofs are based on the techniques of determinantal point
processes.

First, we show that for any  extreme character of $U(\infty)$, the
corresponding random point process on $\Z\times\Z_+$ is
determinantal, and we compute the correlation kernel in the form of
a double contour integral of a fairly simple integrand. This result
(Theorem \ref{sectiontwotheorem}) is similar in spirit to the formula for the
correlation kernel of the Schur process from
\cite{kn:OR2}, but it does not seem to be in
direct relationship with it. After that we perform the asymptotic
analysis of the contour integrals largely following the ideas of
\cite{kn:OK}, \cite{kn:OR2},
\cite{kn:OR}.
\\*[0.1in]
\textbf{Acknowledgements}. The authors are very grateful
to Grigori Olshanski for a number of valuable suggestions. The first
named author (A.~B.) was partially supported by the NSF grant
DMS-0707163.

\section{Description of the Model}
Let $U(N)$ denote the group of all $N\times N$ unitary matrices. For each $N$, $U(N)$ is naturally embedded in $U(N+1)$ as the subgroup fixing the $(N+1)$-th basis vector. Equivalently, each $U\in U(N)$ can be thought of as an $(N+1)\times(N+1)$ matrix by setting $U_{i,N+1}=U_{N+1,j}=0$ for $1\leq i,j\leq N$ and $U_{N+1,N+1}=1$. The union $\cup_{N=1}^{\infty} U(N)$ is denoted $U(\infty)$.

A \textit{character} of $U(\infty)$ is a positive definite function $\chi:U(\infty)\rightarrow\mathbb{C}$ which is constant on conjugacy classes and normalized ($\chi(e)=1$). We further assume that $\chi$ is continuous on each $U(N)\subset U(\infty)$. The set of all characters of $U(\infty)$ is convex, and the extreme points of this set are called \textit{extreme characters}.

The extreme characters of $U(\infty)$ can be parametrized as follows: Let $\mathbb{R}^{\infty}$ denote the product of countably many copies of $\mathbb{R}$. Let $\Omega$ be the set of all $(\alpha^+,\alpha^-,\beta^+,\beta^-,\delta^+,\delta^-)$ such that (\cite{kn:OL3}, \S 1)

\[\alpha^{\pm}=(\alpha_1^{\pm}\geq\alpha_2^{\pm}\geq\ldots\geq 0)\in\mathbb{R}^{\infty}, \ \ \beta^{\pm}=(\beta_1^{\pm}\geq\beta_2^{\pm}\geq\ldots\geq 0)\in\mathbb{R}^{\infty}, \ \ \delta^{\pm}\in\mathbb{R},\]

\[\displaystyle\sum_{i=1}^{\infty} (\alpha_i^{\pm}+\beta_i^{\pm})\leq\delta^{\pm}, \ \ \beta_1^+ + \beta_1^-\leq 1.\]

Set

\[\gamma^{\pm}=\delta^{\pm}-\displaystyle\sum_{i=1}^{\infty} (\alpha_i^{\pm}+\beta_i^{\pm})\geq 0.\]

Each $\omega$ in this set defines a function $\chi^{\omega}$ on $U(\infty)$ by

\[\chi^{\omega}(U)=\displaystyle\prod_{u\in\mathrm{Spectrum}(U)} f_0(u)\]
\begin{align}
&&f_0(u)=e^{\gamma^+(u-1)+\gamma^-(u^{-1}-1)}\displaystyle\prod_{i=1}^{\infty}\frac{1+\beta_i^+(u-1)}{1-\alpha_i^+(u-1)}\frac{1+\beta_i^-(u^{-1}-1)}{1-\alpha_i^-(u^{-1}-1)}.
\label{f_0}
\end{align}

As $\omega$ ranges over $\Omega$, the functions $\chi^{\omega}$ turn out to be all the extreme characters of $U(\infty)$ (\cite{kn:VO}, \cite{kn:VK2}, \cite{kn:OKOL}).

Equipping $\mathbb{R}^{\infty}\times\mathbb{R}^{\infty}\times\mathbb{R}^{\infty}\times\mathbb{R}^{\infty}\times\mathbb{R}\times\mathbb{R}$ with the product topology induces a topology on $\Omega$. For any fixed $U\in U(\infty)$, $\chi^{\omega}(U)$ is a continuous function of $\omega$. For any character $\chi$ of $U(\infty)$, there exists a unique Borel probability measure $P$ on $\Omega$ such that

\[\chi(U)=\int_{\Omega}\chi^{\omega}(U)dP,\]
see \cite{kn:OL3}, Theorem 9.1.
This measure is called the \textit{spectral measure} of $\chi$.

It is a classical result that the irreducible representations of $U(N)$ can be paramterized by nonincreasing sequences $\lambda=(\lambda_1\geq\ldots\geq\lambda_N)$ of $N$ integers (see e.g.~\cite{kn:ZH}). Such sequences are called \textit{signatures (or highest weights) of length N}. Thus there is a natural bijection $\lambda\leftrightarrow\chi^{\lambda}$ between signatures of length $N$ and the conventional irreducible characters of $U(N)$.

The extreme characters of $U(\infty)$ can be approximated by $\chi^{\lambda}$ with growing signatures $\lambda$. To state this precisely we need more notation.

Represent a signature $\lambda$ as a pair of Young diagrams $(\lambda^+,\lambda^-)$, where $\lambda^+$ consists of positive $\lambda_i$'s and $\lambda^-$ consists of negative $\lambda_i$'s. Zeroes can go in either of the two:

\[\lambda=(\lambda_1^+,\lambda_2^+,\ldots,-\lambda_2^-,-\lambda_1^-).\]

Let $d(\cdot)$ denote the number of diagonal boxes of a Young diagram and set $d^+=d(\lambda^+)$ and $d^-=d(\lambda^-)$. Recall that the Frobenius coordinates $p_i,q_i$ of a Young diagram $\lambda$ are defined by

\[p_i=\lambda_i-i, \ \ q_i=(\lambda')_i-i, \ \ 1\leq i\leq d(\lambda),\]
where $\lambda'$ is the transposed diagram.

The dimension of the irreducible representation of $U(N)$ indexed by a signature $\lambda=(\lambda_1,\ldots,\lambda_N)$ is given by Weyl's formula:

\[\dim_N\lambda = \chi^{\lambda}(1,\ldots,1) = \displaystyle\prod_{1\leq i<j\leq N} \frac{\lambda_i-i-\lambda_j+j}{j-i}.\]

Define the \textit{normalized} irreducible characters by

\[\tilde{\chi}^{\lambda}=\frac{1}{\dim_N\lambda}\chi^{\lambda}.\]

Note that $\tilde{\chi}^{\lambda}(e)=1$.

Given a sequence $\{f_N\}$ of functions on $U(N)$, we say that $f_N$'s \textit{approximate} a function $f$ on $U(\infty)$ if for any fixed $N_0$, the restrictions of the functions $f_N$ (for $N\geq N_0$) to $U(N_0)$ uniformly tend, as $N\rightarrow\infty$, to the restriction of $f$ to $U(N_0)$. We have the following approximation theorem:

\begin{theorem}
Let $\chi$ be the extreme character corresponding to $(\alpha^{\pm},\beta^{\pm},\gamma^{\pm})\in\Omega$. Let $\{\lambda(n)\}$ be a sequence of signatures of length $n$ with Frobenius coordinates $p_i^{\pm}(n), q_i^{\pm}(n)$. Then the functions $\tilde{\chi}^{\lambda(n)}$ approximate $\chi$ iff

\[\displaystyle\lim_{n\rightarrow\infty} \frac{p_i^{\pm}(n)}{n}=\alpha_i^{\pm}, \ \ \lim_{n\rightarrow\infty} \frac{q_i^{\pm}(n)}{n}=\beta_i^{\pm}, \ \ \lim_{n\rightarrow\infty} \frac{\vert(\lambda(n))^{\pm}\vert}{n}=\delta^{\pm}\]
for all $i$.
\end{theorem}
\begin{proof}
This theorem is due to Vershik and Kerov \cite{kn:VK2}. See \cite{kn:OKOL} for a detailed proof.
\end{proof}

Let $\mathbb{GT}_N$ be the set of all signatures of length $N$ and set $\mathbb{GT}=\cup_N\mathbb{GT}_N$. Turn $\mathbb{GT}$ into a graph by drawing an edge between signatures $\lambda\in\mathbb{GT}_N$ and $\mu\in\mathbb{GT}_{N+1}$ if $\lambda$ and $\mu$ satisfy the branching relation $\lambda\prec\mu$, where $\lambda\prec\mu$ means that $\mu_1\leq\lambda_1\leq\mu_2\leq\lambda_2\leq\ldots\leq\lambda_N\leq\mu_{N+1}$. $\mathbb{GT}$ is also known as the \textit{Gelfand-Tsetlin graph}.

Each character of $U(\infty)$ defines a probability measure $P_N$ on $\mathbb{GT}_N$. If we restrict the extreme character $\chi^{\omega}$ to $U(N)$, we can write

\begin{align}\label{defPN}
\chi^{\omega}\vert_{U(N)} = \displaystyle\sum_{\lambda\in\mathbb{GT}_N} P_N(\lambda)\tilde{\chi}^{\lambda}.
\end{align}

\begin{definition}\label{Plancherel Measure}
\textit{The measure} $P_N$ \textit{corresponding to the extreme character with} $\alpha^{\pm}=\beta^{\pm}=0$ \textit{and arbitrary} $\gamma^{\pm}\geq 0$ \textit{will be called} the Nth level Plancherel measure with parameters $\gamma^{\pm}$. Denote it by $P_N^{\gamma^+,\gamma^-}$.
\end{definition}

The choice of the term is explained by the analogy with the infinite symmetric group $S(\infty)$. The extreme characters of $S(\infty)$ are parameterized by

\[\{(\alpha,\beta,\gamma)\in\mathbb{R}_+^{\infty}\times\mathbb{R}_+^{\infty}\times\mathbb{R}_+; \sum{(\alpha_i+\beta_i)}+\gamma=1\}.\]

The measure on partitions of $n$ obtained from the character with $\alpha_i=\beta_i=0,\gamma=1$, simiarly to the measure $P_N$ above, assigns the weight $(\dim\lambda)^2/n!$ to a partition $\lambda$ and is commonly called the Plancherel measure. Here $\dim\lambda$ is the dimension of the irreducible representation of $S_n$ corresponding to $\lambda$.

Let $\chi$ be a character of $U(\infty)$ and let $P$ and $P_N$ be its corresponding decomposing measures on $\Omega$ and $\mathbb{GT}_N$. For any $N$, embed $\mathbb{GT}_N$ into $\Omega$ by sending $\lambda$ to $(a^+,a^-,b^+,b^-,c^+,c^-)$ where

\[a_i^{\pm}=\frac{p_i^{\pm}}{N},\ \ b_i^{\pm}=\frac{q_i^{\pm}}{N},\ \ c^{\pm}=\frac{\vert\lambda^{\pm}\vert}{N}.\]

Define a probability measure $\underline{P}_N$  on $\Omega$ to be the pushforward of $P_N$ under this embedding. Then $\underline{P}_N$ weakly converges to $P$ as $N\rightarrow\infty$ (\cite{kn:OL3}, Theorem 10.2).

This implies that as $N\rightarrow\infty$, the Plancherel measures $P_N^{\gamma^+,\gamma^-}$ converge to the delta measure at $\omega=(\alpha_i^{\pm}=\beta_i^{\pm}=0,\gamma^+,\gamma^-)$, that is, the row and column lengths for $\lambda^{\pm}$ distributed according to $P_N^{\gamma^+,\gamma^-}$ grow sublinearly in $N$.

The main goal of this paper is to study the asymptotic behavior of the signatures distributed according to the Plancherel measures $P_N^{\gamma^+,\gamma^-}$ as $N\rightarrow\infty$. We will also study a more general object: the corresponding probability measures on objects called paths in $\mathbb{GT}$.

A \textit{path} in $\mathbb{GT}$ is an infinite sequence $t=(t_1,t_2,\ldots)$ such that $t_i\in\mathbb{GT}_i$ and $t_i\prec t_{i+1}$. Let $\cal T$ be the set of all paths.

We also have \textit{finite paths}, which are sequences $\tau=(\tau_1,\tau_2,\ldots,\tau_N)$ such that $\tau_i\in\mathbb{GT}_i$ and $\tau_1\prec\tau_2\prec\ldots\prec\tau_N$. The set of all paths of length $N$ is denoted by $\cal T$$_N$. For each finite path $\tau\in\cal T$$_N$, let $C_{\tau}$ be the cylinder set
\begin{center}
$C_{\tau}=\{t\in\cal T$$:(t_1,t_2,\ldots,t_N)=\tau\}$.
\end{center}

A character $\chi$ of $U(\infty)$ also defines a probability measure $M^{\chi}$ on $\cal T$ which can be specified by setting
\begin{align}\label{CTAU}
M^{\chi}(C_{\tau})=\frac{P_N(\lambda)}{\dim_N\lambda},
\end{align}
where $P_N$ is as above and $\tau$ is an arbitrary finite path ending at $\lambda$ (\cite{kn:OL3},\S 10). In particular, any $\omega\in\Omega$ defines a measure on $\cal T$ via the corresponding extreme character $\chi^{\omega}$. If $\omega$ satisfies $\alpha^{\pm}_i=\beta^{\pm}_i=0$ with arbitrary $\gamma^{\pm}$, then let $P^{\gamma^+,\gamma^-}$ denote this measure.

\section{Plancherel measures as determinantal point processes}
In order to analyze $P_N^{\gamma^+,\gamma^-}$ and $P^{\gamma^+,\gamma^-}$, it is convenient to represent signatures as finite point configurations (subsets) in one-dimensional lattice. Assign to each signature $\lambda\in\mathbb{GT}_N$ a point configuration $\cal L$$(\lambda)\subset\mathbb{Z}$ by
\begin{center}
$\lambda=(\lambda_1,\ldots,\lambda_N)\mapsto$$\cal L$$(\lambda)=\{\lambda_1-1,\ldots,\lambda_N-N\}$.
\end{center}

The pushforward of $P_N^{\gamma^+,\gamma^-}$ under this map is a measure on subsets of $\mathbb{Z}$, that is, a random point process on $\Z$. Denote this point process by $\cal P$$_N^{\gamma^+,\gamma^-}$. The map $\lambda\mapsto\cal L$$(\lambda)$ can be seen visually. For example, if $\lambda=(4,2,0,0,-1,-3)$, then ${\cal L}(\lambda)=\{3,0,-3,-4,-6,-9\}$. See Figure~\ref{Young configuration}.
\begin{figure}[htp]
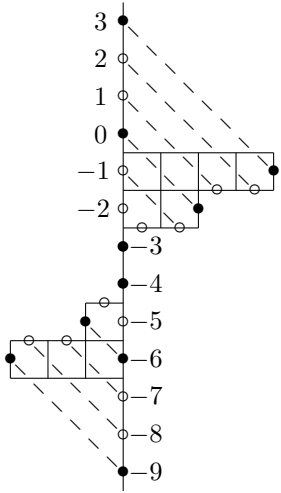

\caption{Black dots represent points in the configuration and white dots represent points not in the configuration.}
\[
\xy
(-3,22.5)*{3};
(-3,17.5)*{2};
(-3,12.5)*{1};
(-3,7.5)*{0};
(-4,2.5)*{-1};
(-4,-2.5)*{-2};
(3,-7.5)*{-3};
(3,-12.5)*{-4};
(3,-17.5)*{-5};
(3,-22.5)*{-6};
(3,-27.5)*{-7};
(3,-32.5)*{-8};
(3,-37.5)*{-9};
(0,5)*{}; (20,5)*{} **\dir{-};
(0,0)*{};  (20,0)*{} **\dir{-};
(0,-5)*{}; (10,-5)*{} **\dir{-};
(0,-10)*{}; (0,-10)*{} **\dir{-};
(0,-15)*{}; (-5,-15)*{} **\dir{-};
(0,-20)*{}; (-15,-20)*{} **\dir{-};
(0,-25)*{}; (-15,-25)*{} **\dir{-};
(0,25)*{}; (0,-40)*{} **\dir{-};
(5,5)*{}; (5,-5)*{} **\dir{-};
(10,5)*{}; (10,-5)*{} **\dir{-};
(15,5)*{}; (15,0)*{} **\dir{-};
(20,5)*{}; (20,0)*{} **\dir{-};
(-5,-15)*{}; (-5,-25)*{} **\dir{-};
(-10,-20)*{}; (-10,-25)*{} **\dir{-};
(-15,-20)*{}; (-15,-25)*{} **\dir{-};
(20,2.5)*{\bullet}; (0,22.5)*{\bullet} **\dir{--};
(17.5,0)*{\circ}; (0,17.5)*{\circ} **\dir{--};
(12.5,0)*{\circ}; (0,12.5)*{\circ} **\dir{--};
(10,-2.5)*{\bullet}; (0,7.5)*{\bullet} **\dir{--};
(7.5,-5)*{\circ}; (0,2.5)*{\circ} **\dir{--};
(2.5,-5)*{\circ}; (0,-2.5)*{\circ} **\dir{};
(0,-7.5)*{\bullet}; (0,-7.5)*{\bullet} **\dir{--};
(0,-12.5)*{\bullet}; (0,-12.5)*{\bullet} **\dir{--};
(-2.5,-15)*{\circ}; (0,-17.5)*{\circ} **\dir{};
(-5,-17.5)*{\bullet}; (0,-22.5)*{\bullet} **\dir{--};
(-7.5,-20)*{\circ}; (0,-27.5)*{\circ} **\dir{--};
(-12.5,-20)*{\circ}; (0,-32.5)*{\circ} **\dir{--};
(-15,-22.5)*{\bullet}; (0,-37.5)*{\bullet} **\dir{--};
\endxy
\]
\label{Young configuration}
\end{figure}

Given a point process on $\Z$, define the \textit{nth correlation function} $\rho_n$ by
\[\rho_n:\Z^n\rightarrow [0,1]\]
\[(x_1,x_2,\ldots,x_n)\mapsto\mathrm{Prob}(\{X\subset\Z:\{x_1,x_2,\ldots,x_n\}\subset X\}).\]
(There is a more general definition of correlation functions, but it will not be needed here. See e.g. \cite{kn:BO}, \S 5 for more details). Clearly, this function is symmetric with respect to the permutations of the arguments.

On a countable discrete state space ($\Z$ in our case) a point process is uniquely determined by its correlation functions (see e.g. \cite{kn:B}, \S 4), so to study the measure it suffices to study its correlation functions.

A point process is \textit{determinantal} if there exists a function $K$ such that
\[\rho_n(x_1,x_2,\ldots,x_n)=\det[K(x_i,x_j)]_{1\leq i,j\leq n} \ \ \mathrm{for} \ \mathrm{any} \ n=1,2,\ldots.\]
The function $K$ is the \textit{correlation kernel}. A useful observation is that $K$ is not unique: $K(x,y)$ and $\frac{f(x)}{f(y)}K(x,y)$ define the same correlation functions for an arbitrary function $f$.

Just as $\lambda\mapsto\cal L$$(\lambda)$ defines a map from $\mathbb{GT}_n$ to the set of subsets of $\Z$, we have a map from the set $\cal T$ of paths in the Gelfant-Tsetlin graph to subsets of $\Z_+\times\Z$. Let $t=(t_1\prec t_2\prec\ldots)$ be a path in $\mathbb{GT}$. Each $t_i$ is a signature of length $i$ which will be written as $\lambda^{(i)}=(\lambda_1^{(i)},\lambda_2^{(i)},\ldots,\lambda_i^{(i)})$. Then map $t$ to
\begin{center}
$\cal L$$(t)=\{(i,\lambda_j^{(i)}-j):1\leq i<\infty,1\leq j\leq i\}\subset\Z_+\times\Z.$
\end{center}

The pushforward of $P^{\gamma^+,\gamma^-}$ under this map will be denoted by ${\cal P}^{\gamma^+,\gamma^-}$. This is a random point process on $\Z_+\times\Z$.

One more introductory concept is needed. Define a map $\Delta$ by
\[\Delta: 2^{\Z_+\times\Z}\rightarrow 2^{\Z_+\times\Z}, \ \ X\mapsto (\Z_+\times\Z)\backslash X.\]
Given a point process ${\cal P}$ on $\Z_+\times\Z$, its pushforward under $\Delta$ is also a point process on $\Z_+\times\Z$, which will be denoted ${\cal P}_{\Delta}$. The map $\Delta$ is often referred to as ``particle-hole involution". With this notation, we have the following proposition:

\begin{proposition}\label{Complement}
If ${\cal P}$ is a determinantal point process with correlation kernel $K(n_i,x_i;n_j,x_j)$, then ${\cal P}_{\Delta}$ is also a determinantal point prcess. Its correlation kernel is $\delta_{n_i=n_j,x_i=x_j}-K(n_i,x_i;n_j,x_j)$.
\end{proposition}
\begin{proof}
See Proposition A.8 of \cite{kn:BOO}.
\end{proof}

Let us now state the main theorem of this section.

\begin{theorem}\label{sectiontwotheorem}
The point process $\cal P$$^{\gamma^+,\gamma^-}$ is determinantal. Let $K(n_i,x_i;n_j,x_j)$ denote its correlation kernel. If $n_1\geq n_2$, then
\[K(n_1,x_1;n_2,x_2)=\left(\frac{1}{2\pi i}\right)^2\oint\oint \frac{e^{\gamma^-u+\gamma^+u^{-1}}u^{x_1}(1-u)^{n_1}}{e^{\gamma^-w+\gamma^+w^{-1}}w^{1+x_2}(1-w)^{n_2}}\frac{dudw}{u-w}.\]
If $n_1<n_2$, then
\begin{multline}\label{uuu}
K(n_1,x_1;n_2,x_2)=-\frac{1}{2\pi i}\oint\frac{z^{x_1-x_2-1}}{(1-z)^{n_2-n_1}}dz\\
+\left(\frac{1}{2\pi i}\right)^2\oint\oint \frac{e^{\gamma^-u+\gamma^+u^{-1}}u^{x_1}(1-u)^{n_1}}{e^{\gamma^-w+\gamma^+w^{-1}}w^{1+x_2}(1-w)^{n_2}}\frac{dudw}{u-w}.
\end{multline}
In these expressions, $u$ is integrated over $\vert u\vert=r<1$ and $w$ is integrated over $\vert w-1\vert=\epsilon<1-r$ and $z$ is integrated over $\vert z\vert=r<1$.
\end{theorem}

%\textbf{Remark.} When performing the asymptotic analysis later in this paper, it is more convenient to express the correlation kernel slightly differently. Namely, if $n_1<n_2$, then
%\[K(n_1,x_1;n_2,x_2)=\left(\frac{1}{2\pi i}\right)^2\oint\oint \frac{e^{\gamma^-u+\gamma^+u^{-1}}u^{x_1}(1-u)^{n_1}}{e^{\gamma^-w+\gamma^+w^{-1}}w^{1+x_2}(1-w)^{n_2}}\frac{dudw}{u-w}\]
%\[+\frac{1}{2\pi i}\oint\frac{z^{x_1-x_2-1}}{(1-z)^{n_2-n_1}}dz\]
%where $u$ is integrated over $\vert u\vert=r<1$, $w$ is integrated over $\vert w-1\vert=\epsilon<1-r$ and $z$ is integrated clockwise over the circle $\vert z\vert=r<1$.

\begin{corollary}\label{sectiontwocorollary}
The point process ${\cal P}_{\Delta}^{\gamma^+,\gamma^-}$ is determinantal. Let $K_{\Delta}(n_i,x_i;n_j,x_j)$ denote its correlation kernel. If $n_1>n_2$, then
\[K_{\Delta}(n_1,x_1;n_2,x_2)=-\left(\frac{1}{2\pi i}\right)^2\oint\oint \frac{e^{\gamma^-u+\gamma^+u^{-1}}u^{x_1}(1-u)^{n_1}}{e^{\gamma^-w+\gamma^+w^{-1}}w^{1+x_2}(1-w)^{n_2}}\frac{dudw}{u-w}.\]
If $n_1\leq n_2$, then
\begin{multline}\label{ttt}
K_{\Delta}(n_1,x_1;n_2,x_2)=\frac{1}{2\pi i}\oint\frac{z^{x_1-x_2-1}}{(1-z)^{n_2-n_1}}dz\\
-\left(\frac{1}{2\pi i}\right)^2\oint\oint \frac{e^{\gamma^-u+\gamma^+u^{-1}}u^{x_1}(1-u)^{n_1}}{e^{\gamma^-w+\gamma^+w^{-1}}w^{1+x_2}(1-w)^{n_2}}\frac{dudw}{u-w}.
\end{multline}
In these expressions, $u$ is integrated over $\vert u\vert=r<1$ and $w$ is integrated over $\vert w-1\vert=\epsilon<1-r$ and $z$ is integrated over $\vert z\vert=r<1$.
\end{corollary}
\begin{proof}
The corollary follows immediately from Proposition~\ref{Complement} and the fact that
\[\delta_{x_1=x_2}=\frac{1}{2\pi i}\oint_{\vert z\vert=r} \frac{z^{x_1-x_2-1}}{(1-z)^{n_2-n_1}}dz\]
for $n_1=n_2$. Note that in Theorem~\ref{sectiontwotheorem} the two cases for the kernel are $n_1\geq n_2$ and $n_1<n_2$, while in Corollary~\ref{sectiontwocorollary} the two cases are $n_1>n_2$ and $n_1\leq n_2$.
\end{proof}

\textbf{Remark.} Let $\bar{K}(n_1,x_1;n_2,x_2)$ and $\bar{K}_{\Delta}(n_1,x_1;n_2,x_2)$ denote the correlation kernels of ${\cal P}^{\gamma^-,\gamma^+}$ and ${\cal P}_{\Delta}^{\gamma^-,\gamma^+}$, respectively ($\gamma^+$ and $\gamma^-$ switched places). The substitutions $u\mapsto u^{-1},w\mapsto w^{-1}$ and further deformation of the contours show that
\[(-1)^{n_1-n_2}K(n_1,-x_1-n_1-1;n_2,-x_2-n_2-1)=\bar{K}(n_1,x_1;n_2,x_2),\]
\[(-1)^{n_1-n_2}K_{\Delta}(n_1,-x_1-n_1-1;n_2,-x_2-n_2-1)=\bar{K}_{\Delta}(n_1,x_1;n_2,x_2).\]
This can be understood independently. Switching $\gamma^+$ and $\gamma^-$ corresponds to switching $\lambda^+$ and $\lambda^-$ in a signature $\lambda$. In terms of ${\cal L}(\lambda)$, this corresponds to replacing $x_i$ with $-x_i-n_i-1$. For example, consider $\lambda=(4,2,0,0,-1,-3)$ from Figure~\ref{Young configuration}. Switching $\lambda^+$ and $\lambda^-$ gives $\bar{\lambda}=(3,1,0,0,-2,-4)$. Then ${\cal L}(\bar{\lambda})=\{2,-1,-3,-4,-7,-10\}$, which can be obtained from ${\cal L}(\lambda)$ by replacing $x_i$ with $-x_i-6-1$.

\textbf{Remark.} The arguments below actually prove a more general statement. If we define a point process of $\Z_+\times\Z$ similarly to $\cal P$$^{\gamma^+,\gamma^-}$, but starting from an extreme character of $U(\infty)$ with arbitrary parameters $(\alpha_i^{\pm},\beta_i^{\pm},\gamma^{\pm})$, then this process is determinantal and its kernel has a similar form.
%\[K(n_1,x_1;n_2,x_2)=\left(\frac{1}{2\pi i}\right)^2\oint\oint \frac{f_0(u^{-1})u^{x_1}(1-u)^{n_1}}{f_0(w^{-1})w^{1+x_2}(1-w)^{n_2}}\frac{dudw}{u-w},\]
The only change is replacing $E(z)$ below by $f_0(z)$ from equation (\ref{f_0}).

In what follows we use the notation
\[E(z)=e^{\gamma^+(z-1)+\gamma^-(z^{-1}-1)}=e^{-\gamma^+-\gamma^-}e^{\gamma^+z+\gamma^-z^{-1}}.\]
\begin{lemma}\label{Bi-orthogonal ensemble}
Suppose $\lambda=(\lambda_1,\lambda_2,\ldots,\lambda_N)\in\mathbb{GT}_N$. Write $x_k$ for $\lambda_k-k$. Then

\[P_N^{\gamma^+,\gamma^-}(\lambda)=\mathrm{const}\cdot\det[f_j(x_k)]_{1\leq j,k\leq N}\det[g_j(x_k)]_{1\leq j,k\leq N}\]
where
\begin{align}
f_j(x_k)&=\frac{1}{2\pi i}\oint_{\vert u\vert=1} E(u)u^{-1-x_k-j}du, &\ 1\leq j\leq N.\\
g_j(x_k)&=x_k^{j-1}, &\ 1\leq j\leq N.
\end{align}
\end{lemma}
\begin{proof}
Writing $E(u)=\displaystyle\sum_{l=-\infty}^{\infty} c(l)u^l$ and integrating $E(u)u^k$ over the unit circle, we can solve for $c$ to get

\[c(l)=\frac{1}{2\pi i}\oint_{\vert u\vert=1} E(u)u^{-1-l}du\]

Set $\omega=(\alpha_i^{\pm}=\beta_i^{\pm}=0,\gamma^+,\gamma^-)$. For $U\in U(N)$ with spectrum ${u_1,\ldots,u_N}$, we can write $\chi^{\omega}(U)=E(u_1)\ldots E(u_N)$. Recall that $P_N^{\gamma^+,\gamma^-}$ is defined by (\ref{defPN}). Using (\cite{kn:OL3}, Lemma 6.5), we can express $\chi^{\omega}\vert_{U(N)}$ as $\displaystyle\sum_{\lambda\in \mathbb{GT}_N} c(\lambda)\chi^{\lambda}$, where

\[c(\lambda)=c(\lambda_1,\ldots,\lambda_N)=\det[c(\lambda_k-k+j)]_{1\leq j,k\leq N}.\]

Set $f_j(x_k)=c(x_k+j)$. Since $\chi^{\lambda}=\tilde{\chi}^{\lambda}\cdot\dim_N\lambda$, with
\[\dim_N \lambda=\prod_{1\leq i<j\leq N} \frac{\lambda_i-i-\lambda_j+j}{j-i}=\mathrm{const}\cdot\prod_{1\leq i<j\leq N} \left((\lambda_i-i)-(\lambda_j-j)\right),\]
we get the additional Vandermonde determinant $\det[(\lambda_k-k)^{j-1}]=\det[x_k^{j-1}]$.
\end{proof}

\textbf{Remark.} Observe that the argument above and (\ref{CTAU}) imply that $P^{\gamma^+,\gamma^-}(C_\tau)=\det[f_j(x_k)]_{1\leq j,k\leq N}$.

To state the next result we need slightly different notation. Let
\begin{center}
$\cal P$$^{\gamma^+,\gamma^-}(\{x_k^{(n)}:1\leq n\leq N,\ 1\leq k\leq n\})=P^{\gamma^+,\gamma^-}(C_{\tau})$
\end{center}
if there exists a path $\tau=(\lambda^{(1)}\prec\ldots\prec\lambda^{(n)})$ such that
\[\lambda^{(n)}=(x_1^{(n)}+1,x_2^{(n)}+2,\ldots,x_n^{(n)}+n),\]
and $\cal P$$^{\gamma^+,\gamma^-}(\{x_k^{(n)}\})=0$ otherwise.
\begin{proposition}\label{Lemma 3.4}
Let $\{x_k^{(n)}:1\leq n\leq N, 1\leq k\leq n\}$ be arbitrary integers satisfying $x_k^{(n)}\geq x_{k+1}^{(n)}$ for all $n,k$. Then
\begin{center}
$\cal P$$^{\gamma^+,\gamma^-}(\{x_k^{(n)}:1\leq n\leq N,\ 1\leq k\leq n\})=\displaystyle\mathrm{const}\cdot\prod_{n=1}^{N-1}\det[\phi_n(x_i^{(n)},x_j^{(n+1)})]_{1\leq i,j\leq n+1}\det[f_i(x_j^{(N)})]_{1\leq i,j\leq N}$
\end{center}
where $x^{(n)}_{n+1}$ are virtual variables\footnote{One can think of virtual variables as being equal to negative infinity.}, and $\phi_n$ is defined by

$$
\phi_n(x,y) :=
\begin{cases}
1 & \text{\ \ if } x\leq y, \\
1 & \ \ x \text{\ \ virtual},\\
0 & \text{\ \ otherwise.}
\end{cases}
$$
\end{proposition}
\begin{proof}
By the remark after \ref{Bi-orthogonal ensemble}, it suffices to prove that $\prod\det[\phi_n]$ acts as a indicator function. It takes the value of $1$ if $\lambda^{(1)}\prec\lambda^{(2)}\prec\ldots\prec\lambda^{(N)}$ and $0$ otherwise, where
\[\lambda^{(n)}=(x_1^{(n)}+1,\ldots,x_n^{(n)}+n).\]

If $\lambda^{(1)}\prec\lambda^{(2)}\prec\ldots\prec\lambda^{(N)}$, so that $x_1^{(n+1)}\geq x_1^{(n)}>x_2^{(n+1)}\geq x_2^{(n)}>\ldots\geq x_n^{(n)}>x_{n+1}^{(n+1)}$ for each $n$, then

$$
\phi_n(x_i^{(n)},x_j^{(n+1)}) =
\begin{cases}
1 & \text{\ \ if } j\leq i,\\
0 & \text{\ \ if } i<j.
\end{cases}
$$
So $\det[\phi_n]=1$ for each $n$.

Conversely, suppose that $\prod\det[\phi_n]=1$, so that $\det[\phi_n]\neq 0$ for each $n$. Notice that the matrix $[\phi_n(x_i^{(n)},x_j^{(n+1)})]$ consists entirely of zeroes and ones. Also notice that the number of ones in the $k$th column is greater than or equal to the number of ones in the $j$th column for $k<j$. Additionally, if the $(i,j)$ entry is zero then so is the $(i-1,j)$ entry. Since the determinant is nonzero, this means that no two columns are equal, so each column must have a different number of ones, so the $(i,j)$ entry is $1$ if $j\leq i$ and $0$ if $i<j$. This says exactly that $\lambda^{(1)}\prec\lambda^{(2)}\prec\ldots\prec\lambda^{(N)}$, and each determinant in the product is equal to $1$.
\end{proof}
We can now prove Theorem~\ref{sectiontwotheorem}.

\textit{Proof of Theorem~\ref{sectiontwotheorem}}.
For computational purposes, it is actually easier to consider
$$
\phi_n(x,y) :=
\begin{cases}
\theta_n^{x-y} & \text{\ \ if } x\leq y, \\
\theta_n^{-y} & \ \ x \text{\ \ virtual},\\
0 & \text{\ \ otherwise.}
\end{cases}
$$
with mutually distinct $\theta_n$'s and then take $\theta_n\rightarrow 1$. It is also convenient to denote $\theta_0=1$. We will assume that $\vert\theta_n\vert>\vert\theta_{n-1}\vert>1$ for all $n$. Notice that $\det[f_i(x_j^{(N)})]$ only depends on the linear span of $f_1,\ldots,f_N$ (up to a constant), so redefine

\[f_j(x)=\frac{1}{2\pi i}\oint_{\vert u\vert=const} E(u)u^{-2-x}p_{j-1}(u^{-1})du,\ \mathrm{where}\]
\[p_{j-1}(x)=(\theta_0-x)\ldots\widehat{(\theta_{j-1}-x)}\ldots(\theta_{N-1}-x)=\displaystyle\prod_{k=0,k\neq j-1}^{N-1}(\theta_k-x).\]

The rest of the proof is a direct application of Lemma 3.4 of~\cite{kn:BPFS}, where we use the notation $\Psi^N_{N-j}=f_j$.

Taking the Fourier Transform of $\phi_n$, we obtain
\[\phi_n(x,y)=\frac{1}{2\pi i}\oint_{\vert z\vert=1} F_n(z)z^{x-y-1}dz\]
\begin{align}\label{phin}
\phi^{(n_1,n_2)}(x,y)=\frac{1}{2\pi i}\oint_{\vert z\vert=1} F_{n_1}(z)\ldots F_{n_2-1}(z)z^{x-y-1}dz
\end{align}
where $F_n(z)=(1-\theta_n^{-1}z)^{-1}$ and $n_1<n_2$. We also agree that $\phi^{(n_1,n_2)}\equiv 0$ if $n_1\geq n_2$. In case $x$ is a virtual variable (which is denoted by \textit{virt}), then

\begin{align*}
\phi^{(n_1,n_2)}(virt,y) =& \sum_{m\in\mathbb{Z}}\phi_{n_1}(virt,m)\phi^{(n_1+1,n_2)}(m,y)\\
%=& \sum_{m=-\infty}^y \theta_{n_1}^m\cdot\frac{1}{2\pi i}\oint F_{n_1+1}(z)\ldots F_{n_2-1}(z)z^{m-y-1}dz\\
%=&\frac{1}{2\pi i}\oint F_{n_1+1}(z)\ldots F_{n_2-1}(z)z^{-y-1}\frac{(\theta_{n_1}z)^y}{1-\theta_{n_1}^{-1}z^{-1}}dz\\
=&\frac{\theta_{n_1}^{-y}}{2\pi i}\oint_{\vert\theta_{n_1}\vert<\vert z\vert=const<\vert\theta_{n_1+1}\vert} F_{n_1+1}(z)\ldots F_{n_2-1}(z)\frac{dz}{z-\theta_{n_1}}\\
=&\theta_{n_1}^{-y} F_{n_1+1}(\theta_{n_1})\ldots F_{n_2-1}(\theta_{n_1})\\
\end{align*}

This allows us to calculate the matrix $M$ (cf. \cite{kn:BPFS}, Lemma 3.4). In the following equation, $\Gamma(r_1,r_2)$ denotes the boundary of an annulus of radii $r_1<r_2$ in the complex plane.

\begin{align*}
M_{ij} =& (\phi^{(i-1,N)}*\Psi^N_{N-j})(virt)=\sum_{y\in \mathbb{Z}}\phi^{(i-1,N)}(virt,y)\Psi^N_{N-j}(y)\\
=& \sum_{y\in \mathbb{Z}}\theta_{i-1}^{-y}F_i(\theta_{i-1})\ldots F_{N-1}(\theta_{i-1})\cdot\frac{1}{2\pi i}\oint_{\vert u\vert=1}E(u)u^{-2-y}p_{j-1}(u^{-1})du\\
=&-F_i(\theta_{i-1})\ldots F_{N-1}(\theta_{i-1})\frac{1}{2\pi i}\oint_{\Gamma(r,1),r<\vert\theta_{i-1}\vert^{-1}} E(u)u^{-2}p_{j-1}(u^{-1})\frac{u\theta_{i-1}}{1-u\theta_{i-1}}du\\
=&F_i(\theta_{i-1})\ldots F_{N-1}(\theta_{i-1})E(\theta_{i-1}^{-1})p_{j-1}(\theta_{i-1})\theta_{i-1}.\\
\end{align*}

Notice that $M$ is diagonal because $p_{j-1}(\theta_{i-1})=0$ unless $i=j$. We have one more preliminary calcuation (cf. \cite{kn:BPFS}, formula (3.22)):
\begin{align*}
&\Psi^n_{n-j}(x)=\sum_{y\in\mathbb{Z}} \phi^{(n,N)}(x,y)\Psi^N_{N-j}(y)\\
=& \left(\frac{1}{2\pi i}\right)^2\oint_{\vert z\vert=1} F_n(z)\ldots F_{N-1}(z)z^{x-1}dz\oint_{\vert u\vert=R>1} E(u)u^{-2}p_{j-1}(u^{-1})\sum_{y\geq x} (zu)^{-y}du\\
=& \left(\frac{1}{2\pi i}\right)^2\oint_{\vert z\vert=1} F_n(z)\ldots F_{N-1}(z)z^{-1}dz\oint_{\vert u\vert=R>1} E(u)u^{-2-x}p_{j-1}(u^{-1})\frac{du}{1-(zu)^{-1}}\\
=& \frac{1}{2\pi i}\oint_{\vert u\vert=1} F_n(u^{-1})\ldots F_{N-1}(u^{-1}) E(u) u^{-2-x} p_{j-1}(u^{-1})du
\end{align*}

We can now calculate $K$ according to \cite{kn:BPFS}, formula (3.26). For $n_1<n_2$,

$$
\begin{aligned}
&K(n_1,x_1;n_2,x_2)+\phi^{(n_1,n_2)}(x_1,x_2)\\
=&\displaystyle\sum_{k=1}^{n_2} [M^{-1}]_{kk}\Psi^{n_1}_{n_1-k}(x_1)\phi^{(k-1,n_2)}(virt,x_2)\\
=&\frac{1}{2\pi i}\oint_{\vert u\vert=1} F_{n_1}(u^{-1})\ldots F_{N-1}(u^{-1})E(u)u^{-2-x_1}\\
&\ \ \ \times\sum_{k=1}^{n_2}\frac{\theta_{k-1}^{-x_2}F_k(\theta_{k-1})\ldots F_{n_2-1}(\theta_{k-1})p_{k-1}(u^{-1})}{F_k(\theta_{k-1})\ldots F_{N-1}(\theta_{k-1})E(\theta_{k-1}^{-1})p_{k-1}(\theta_{k-1})\theta_{k-1}}du\\
=&\frac{1}{2\pi i}\oint_{\vert u\vert=1} F_{n_1}(u^{-1})\ldots F_{N-1}(u^{-1})E(u)u^{-2-x_1}\\
&\ \ \ \times\sum_{k=1}^{n_2}\frac{\theta_{k-1}^{-x_2}p_{k-1}(u^{-1})}{F_{n_2}(\theta_{k-1})\ldots F_{N-1}(\theta_{k-1})E(\theta_{k-1}^{-1})p_{k-1}(\theta_{k-1})\theta_{k-1}}du\\
=&\frac{1}{2\pi i}\oint_{\vert u\vert=1}u^{-2-x_1}\\
&\ \ \ \ \times\biggl(\sum_{k=1}^{n_1}\frac{\theta_{n_1}\ldots\theta_{N-1}E(u)\prod_{l=0,l\neq k-1}^{n_1-1} (\theta_l-u^{-1}) }{\theta_{n_2}\ldots\theta_{N-1}E(\theta_{k-1}^{-1})\prod_{l=0,l\neq k-1}^{n_2-1} (\theta_l-\theta_{k-1})}\theta_{k-1}^{-1-x_2}\\
&\ \ \ \
+\sum_{k=1+n_1}^{n_2}\frac{\theta_{n_1}\ldots\theta_{N-1}E(u)\prod_{l=0}^{n_1-1}(\theta_l-u^{-1})}{(\theta_{k-1}-u^{-1})\theta_{n_2}\ldots\theta_{N-1}E(\theta_{k-1})\prod_{l=0,l\neq k-1}^{n_2-1}(\theta_l-\theta_{k-1})}\theta_{k-1}^{-1-x_2}\biggr)du\\
\end{aligned}
$$
and for $n_1\geq n_2$ the last sum is omitted.

We can write the expression in parantheses as a contour integral that goes around all the $\theta_j$, so we get

\[\left(\frac{1}{2\pi i}\right)^2\oint_{\vert u\vert=r^{-1}>1}\oint_{\vert z-1\vert=\epsilon}\frac{(\theta_0-u^{-1})\ldots(\theta_{n_1-1}-u^{-1})E(u)u^{-2-x_1}}{(\theta_0-z)\ldots(\theta_{n_2-1}-z)E(z^{-1})z^{1+x_2}}\frac{\theta_{n_2}\ldots\theta_{n_1-1}}{(u^{-1}-z)}dudz,\]
assuming that $\vert\theta_n-1\vert<\epsilon$ for all $n$.
Substituting $u\rightarrow u^{-1}$ gives
\[\left(\frac{1}{2\pi i}\right)^2\oint_{\vert u\vert=r<1}\oint_{\vert z-1\vert=\epsilon}\frac{(\theta_0-u)\ldots(\theta_{n_1-1}-u)E(u^{-1})u^{x_1}}{(\theta_0-z)\ldots(\theta_{n_2-1}-z)E(z^{-1})z^{1+x_2}}\frac{\theta_{n_2}\ldots\theta_{n_1-1}}{(u-z)}dudz.\]

There is also the term $-\phi^{(n_1,n_2)}(x_1,x_2)$ from (\ref{phin}), which equals
\[-\left(\frac{1}{2\pi i}\right)^2\oint_{\vert z\vert=const<1}\frac{z^{x_1-x_2-1}}{(1-\theta_{n_1}^{-1}z)\ldots(1-\theta_{n_2-1}^{-1}z)}dz\]
if $n_1<n_2$ and $0$ if $n_1\geq n_2$. Finally, taking all the $\theta_j$ to be $1$ yields the result.

\section{Limits}
\subsection{Limit Shape}
Represent $\lambda\in\mathbb{GT}_N$ as a pair of Young diagrams $(\lambda^+,\lambda^-)$. Figure~\ref{boundary} gives an example with $\lambda=(4,2,0,0,-1,-3),\lambda^+=(4,2),\lambda^-=(3,1)$.
\begin{figure}[htp]
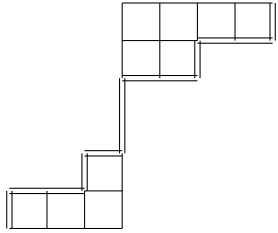

\caption{The double lines show the boundary.}
\[
\xy
(0,5)*{}; (20,5)*{} **\dir{-};
(0,0)*{};  (10,0)*{} **\dir{-};
(10,0)*{};  (20,0)*{} **\dir{=};
(0,-5)*{}; (10,-5)*{} **\dir{=};
(0,-10)*{}; (0,-10)*{} **\dir{-};
(-5,-15)*{}; (0,-15)*{} **\dir{=};
(0,-20)*{}; (-5,-20)*{} **\dir{-};
(-15,-20)*{}; (-5,-20)*{} **\dir{=};
(0,-25)*{}; (-15,-25)*{} **\dir{-};
(0,5)*{}; (0,-5)*{} **\dir{-};
(0,-15)*{}; (0,-5)*{} **\dir{=};
(0,-15)*{}; (0,-25)*{} **\dir{-};
(5,5)*{}; (5,-5)*{} **\dir{-};
(10,5)*{}; (10,0)*{} **\dir{-};
(10,-5)*{}; (10,0)*{} **\dir{=};
(15,5)*{}; (15,0)*{} **\dir{-};
(20,0)*{}; (20,5)*{} **\dir{=};
(-5,-20)*{}; (-5,-25)*{} **\dir{-};
(-5,-20)*{}; (-5,-15)*{} **\dir{=};
(-10,-20)*{}; (-10,-25)*{} **\dir{-};
(-15,-25)*{}; (-15,-20)*{} **\dir{=};
\endxy
\]
\label{boundary}
\end{figure}
We have the following conjecture:

Regard $\lambda\in\mathbb{GT}_N$ as random objects on the probability space $(\mathbb{GT}_N,P_N^{\gamma^+,\gamma^-})$. As $N\rightarrow\infty$, the boundaries of the two Young diagrams, scaled by $N^{-1/2}$, tend to (nonrandom) limit curves. Both limit curves coincide with the limit curve arising from the Plancherel measure on symmetric groups.

Our results strongly suggest that this statement holds, see \S 3.2.

The conditions $\alpha_i^{\pm}=\beta_i^{\pm}=0$ tell us that for fixed $\gamma^{\pm}$ every row and column length grows sublinearly in $N$ (see the end of \S 1). Furthermore, since $\gamma^{\pm}$ correspond to the area of the Young diagrams $\lambda^{\pm}$ (see \S 1), this suggests a scaling of $N^{-1/2}$. See Figure~\ref{Limit Curves}.

\begin{figure}[htp]
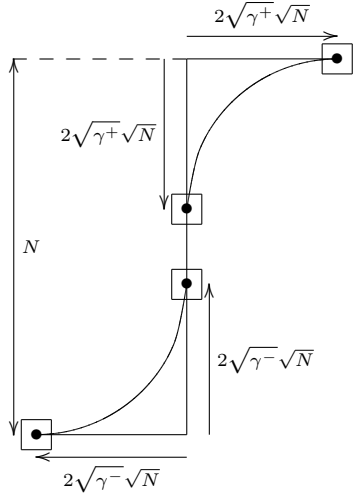

\caption{A visual representation} %of Corollary~\ref{OneLevel}
\[
\xy
(20,0)*{\bullet};
(0,-20)*{\bullet};
(0,-30)*{\bullet};
(-20,-50)*{\bullet};
(0,0)*{}; (20,0)*{} **\dir{-};
(0,0)*{}; (0,-50)*{} **\dir{-};
(0,-50)*{}; (-20,-50)*{} **\dir{-};
(0,-20)*{}="A";
(20,0)*{}="B";
"A"; "B" **\crv{(1,-13.755)&(2,-11,2822)&(3,-9,464)&(4,-8)&(5,-6.77)&(6,-5.717)&(7,-4.801)&(8,-4)&(9,-3.297)&(10,-2.68)&(11,-2.14)&(12,-1.67)&(13,-1.265)&(14,-0.921)&(15,-0.636)&(16,-0.404)&(17,-0.226)&(18,-0.1)&(19,-0.025)};
(0,-30)*{}="C";
(-20,-50)*{}="D";
"C"; "D" **\crv{(-1,-36.245)&(-2,-38.718)&(-3,-40.536)&(-4,-42)&(-5,-43.229)&(-6,-44.283)&(-7,-45.199)&(-8,-46)&(-9,-46.7033)&(-10,-47.3205)&(-11,-47.8606)&(-12,-48.3303)&(-13,-48.735)&(-14,-49.0788)&(-15,-49.3649)&(-16,-49.5959)&(-17,-49.7737)&(-18,-49.8997)&(-19,-49.975)};
{\ar_{2\sqrt{\gamma^+}\sqrt{N}} (-3,0)*{}; (-3,-20)*{}};
{\ar^{2\sqrt{\gamma^+}\sqrt{N}} (0,3)*{}; (20,3)*{}};
{\ar_{2\sqrt{\gamma^-}\sqrt{N}} (3,-50)*{}; (3,-30)*{}};
{\ar^{2\sqrt{\gamma^-}\sqrt{N}} (0,-53)*{}; (-20,-53)*{}};
%{\ar^{N} (-23,0)*{}; (-23,-50)*{}};
{\ar@{<->}^{N} (-23,0)*{}; (-23,-50)*{}};
(-23,0)*{}; (0,0)*{} **\dir{--};
(22,2)*{}; (18,2)*{} **\dir{-};
(18,2)*{}; (18,-2)*{} **\dir{-};
(18,-2)*{}; (22,-2)*{} **\dir{-};
(22,-2)*{}; (22,2)*{} **\dir{-};
(2,-18)*{}; (-2,-18)*{} **\dir{-};
(-2,-18)*{}; (-2,-22)*{} **\dir{-};
(-2,-22)*{}; (2,-22)*{} **\dir{-};
(2,-22)*{}; (2,-18)*{} **\dir{-};
(2,-28)*{}; (-2,-28)*{} **\dir{-};
(-2,-28)*{}; (-2,-32)*{} **\dir{-};
(-2,-32)*{}; (2,-32)*{} **\dir{-};
(2,-32)*{}; (2,-28)*{} **\dir{-};
(-18,-48)*{}; (-22,-48)*{} **\dir{-};
(-22,-48)*{}; (-22,-52)*{} **\dir{-};
(-22,-52)*{}; (-18,-52)*{} **\dir{-};
(-18,-52)*{}; (-18,-48)*{} **\dir{-};
\endxy
\]
\label{Limit Curves}
\end{figure}

Furthermore, we see from Figure~\ref{Young configuration} that vertical segments of the boundary correspond to points in the configuration, while horizontal segments correspond to points not in the configuration. This implies that the first correlation function $\rho_1(x)$ (also known as the density function) corresponds to the density of vertical segments in the boundary. For example, in between the two curves in Figure~\ref{Limit Curves}, the vertical segments are densely packed, so $\rho_1(x)$ should converge to $1$. Above the top curve (the boundary of $\lambda^+$) and below the bottom curve (the boundary of $\lambda^-$), the horizontal segments are densely packed, so $\rho_1(x)$ should converge to $0$. We will see that this is indeed the case.

Notice that near the edges of the Young diagrams (the boxes in Figure~\ref{Limit Curves}), the probability of finding a vertical segment tends to 0 or 1. This means that the vertical segments (or horizontal segments) become so rare that they occur infinitely far away from each other. In other words, for any fixed $k$, the differences $\lambda^{\pm}_k-\lambda^{\pm}_{k+1}$ and $(\lambda^{\pm})'_k-(\lambda^{\pm})'_{k+1}$ both go to infinity as $N\rightarrow\infty$. In fact, we find that $\lambda^{\pm}_k-\lambda^{\pm}_{k+1}$ and $(\lambda^{\pm})'_k-(\lambda^{\pm})'_{k+1}$ are of order $N^{1/6}$. The limiting distribution of $\lambda^{\pm}_k-\lambda^{\pm}_{k+1}$ or $(\lambda^{\pm})'_k-(\lambda^{\pm})'_{k+1}$ normalized by $N^{1/6}$ is referred to as the \textit{edge scaling limit}. We will later prove that the well known Airy determinantal point process appears in the edge limits. On the other hand, if we zoom in at any other point on the limit curves, the behavior there is different. At these points, the differences between consecutive rows and columns stay finite. Their limiting distributions are described by the \textit{bulk limit}. We prove that it coincides with the discrete sine determinantal process. The limit density function in the bulk predicts the limit shape.

We should also consider what happens to the more general object -- the corresponding measure on the set $\tau$ of paths in $\mathbb{GT}$ (see \S 1). Consider two signatures on such a path at levels $n_1$ and $n_2$. If $n_1-n_2$ stays bounded then the bulk and the edge limits of these two signatures are indistinguishable (the local point configurations are essentially the same). However, as $n_2-n_1$ grows, we may see nontrivial joint distributions. It turns out that the proper level scaling in the bulk is $n_1-n_2\sim\sqrt{N}$ while at the edge it is $n_1-n_2\sim N^{2/3}$. We will compute the corresponding scaling limits of the correlation functions later.
%There are two basic limiting cases. One case is to let $x_i$ and $n_i$ vary as constants times $N$ in the correlation kernel $K(n_1,x_1;n_2,x_2)$; these are called \textit{bulk limits}. Another case is to zoom in at certain points (shown by the boxes in Figure~\ref{Limit Curves}) and rescale by a factor of $N^{\alpha}$; these are called {edge limits}. Furthermore, if we let $\gamma^{\pm}$ vary as constants times $N$, then the two curves merge. With $\gamma^{\pm}$ dependence on $N$, again we have bulk and edge limits.

\begin{figure}[htp]
\centering \caption{The limit curves for various values of $a$ and
$b$. The top curve occurs when $a=\frac{1}{25},b=\frac{1}{15}$, the
middle curve occurs when $a=b=\frac{1}{8}$, the bottom curve occurs
when $a=\frac{1}{4},b=\frac{1}{3}$.}

\fbox{\includegraphics[totalheight=0.25\textheight]{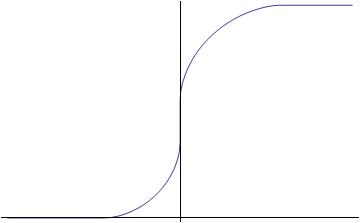}}

\vspace{0.1in}

\fbox{\includegraphics[totalheight=0.25\textheight]{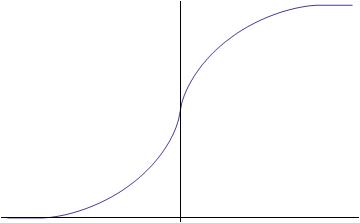}}

\vspace{0.1in}

\fbox{\includegraphics[totalheight=0.25\textheight]{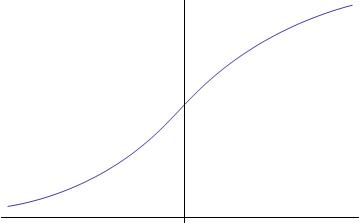}}
\label{limitcurves2}
\end{figure}

It is also interesting to consider the case when the parameters $\gamma^{\pm}$ depend on $N$. If $\gamma^{\pm}$ depend on $N$ in such a way that $\gamma^{\pm}N\rightarrow a>0$, then the areas of the Young diagrams $\lambda^{\pm}$ stay finite. More precisely, we obtain two independent copies of the Poissonized Plancherel measure for symmetric groups.

Additionally, consider what happens when $\gamma^{\pm}$ depend on $N$ in such a way that $\gamma^+/N\rightarrow a>0$ and $\gamma^-/N\rightarrow b>0$ as $N\rightarrow\infty$. The Young diagrams are now scaled by $N^{-1}$. The new hypothetical limit shape depends on the values of $a$ and $b$. See Figure~\ref{limitcurves2}.

The edges of the limit curves correspond to the real roots of a fourth degree polynomial
\[Q_{a,b}(z)=p_0+p_1\left(z+\frac{1}{2}\right)+p_2\left(z+\frac{1}{2}\right)^2+p_3\left(z+\frac{1}{2}\right)^3+16\left(z+\frac{1}{2}\right)^4,\]
\[p_0=1-12(a+b) + 4(a^2+b^2) + 184ab - 256ab(a+b) + 64ab(a-b)^2,\]
\[p_1=8(b-a)(7-2a-2b+16ab),\]
\[p_2=8(2(a+b)^2-10(a+b)-1),\ \ p_3=32(b-a).\]
The expression $Q_{a,b}(c)$ is the discriminant of a simpler polynomial
\[R_{a,b,c}(z)=-bz^3+(b-c-1)z^2+(c+a)z-a.\]
For small $a$ and $b$, $Q_{a,b}$ has four real roots. As $a$ and $b$ increase, two of the real roots become closer until they merge into a double root. For larger values of $a$ and $b$, $Q_{a,b}(z)$ has two real roots.

We will be able to find what values of $a$ and $b$ lead to $Q_{a,b}$ having exactly three distinct real roots (the middle root is a double root). This corresponds to the situation when the two limit curves just barely merge (see the middle image in Figure~\ref{limitcurves2}). The correct scaling there is to let $(\lambda^{\pm})'_i-(\lambda^{\pm})'_{i+1}\sim N^{1/4}$ and $n_1-n_2\sim N^{1/2}$, which results in the Pearcey determinantal process appearing in the limit. At the other edges, letting $\lambda_i-\lambda_{i+1}\sim N^{1/3}$ or $(\lambda^{\pm})'_i-(\lambda^{\pm})'_{i+1}$ and $n_1-n_2\sim N^{2/3}$ results in the Airy process appearing. Away from the edges we still observe the bulk limit.

We now proceed to computing the (scaling) limits of our determinantal point process $\cal P$$^{\gamma^+,\gamma^-}$ corresponding to the limit regimes described above.
\subsection{\texorpdfstring{Limits with $\gamma^{\pm}\propto 1/N$}{Bulk Limits with gamma proportional to 1/N}}
Introduce the kernel $\mathbb{J}$ on $\mathbb{R}_+\times\Z$ by
\[\mathbb{J}(s,x;t,y)=\left(\frac{1}{2\pi i}\right)^2\oint\oint \frac{e^{u^{-1}-tu-w^{-1}+sw}}{w-u}\frac{dudw}{w^{x+1}u^{-y}}\]
where the $w$ and $u$ contours go counterclockwise around $0$ in such a way that the $w$-contour contains the $u$-contour if $s\geq t$, and the $w$-contour is contained in the $u$-contour if $s<t$.
This kernel for $s=t$ is equivalent to the discrete Bessel kernel $\mathbb{K}_{\mathrm{Bessel}}$, which appears when analyzing Plancherel measures for symmetric groups (see e.g., \S 2.4 of \cite{kn:OK}). Additionally, $\mathbb{J}$ is a special case of the kernel (\cite{kn:BO3}, (3.3)) corresponding to $\theta(t)=e^{-2t}$.
%\[K_{C}(s,x;t,y)=\frac{e^{\frac{1}{2}(s-t)}}{(2\pi i)^2}\int\int\frac{e^{\sqrt{\theta(t)}(u^{-1}-u)+\sqrt{\theta(s)}(w-w^{-1})}}{e^{s-t}w-u}\frac{dudw}{w^{x+\frac{1}{2}}u^{-y+\frac{1}{2}}},\]
%and $\theta(t)=e^{-2t}$, then
%\[\mathbb{J}(e^{-2s},x;e^{-2t},y)=e^{ty-sx}K_C(s,x;t,y).\]
\begin{theorem}\label{3.1}
Let $x_1,\ldots,x_k$ be finite and constant. Let $n_1,\ldots,n_k$ and $\gamma^{\pm}$ depend on $N$ in such a way that $n_j/N\rightarrow t_j$ and $\gamma^{\pm}N\rightarrow a>0$. Then as $N\rightarrow\infty$,
\[\det[K(n_i,x_i;n_j,x_j)]_{1\leq i,j\leq k}\rightarrow\det[\mathbb{J}(at_i,x_i;at_j,x_j)]_{1\leq i,j\leq k}\]
\end{theorem}
\begin{proof}
We use the integral representation for the kernel in Theorem~\ref{sectiontwotheorem}.

We first focus our attention on the double integral in $u$ and $w$. Since the integrand is holomorphic everywhere except at $u=0$, $w=1$, $w=u$ and $w=0$, we can deform the contours of integration as shown in Figure~\ref{Deformed}.

\begin{figure}[htp]
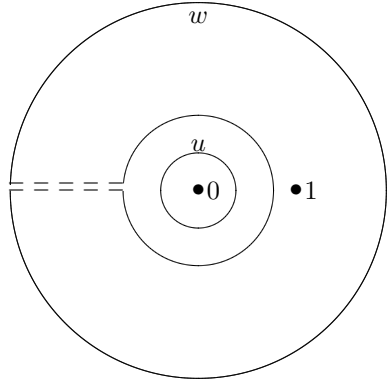

\caption{Deformation of the contours.}
\[
\xy
(0,0)*\xycircle(5,5){-};
(0,0)*\ellipse(10,10)__,=:a(-4){-};
(0,0)*\ellipse(25,25)__,=:a(-2){-};
(0,0)*{\bullet};
(2,0)*{0};
(13,0)*{\bullet};
(15,0)*{1};
(0,6)*{u};
(0,23)*{w};
(-10,0)*{}; (-25,0)*{} **\dir{--};
(-10,1)*{}; (-25,1)*{} **\dir{--};
\endxy
\]
\vspace{1in}
\label{Deformed}
\end{figure}

As $N\rightarrow\infty$, the integrand converges to $0$ for $\vert w\vert$ large enough because $\vert 1-w\vert\gg\vert 1-u\vert$. Therefore we can ignore the outer half of the $w$ contour. Then the contours of integration can be deformed to $\vert u\vert=a/N$ and $\vert w\vert=2a/N$. Making the substitutions $u'=Nu/a$ and $w'=Nw/a$, the double integral is now

\begin{align*}
&\frac{1}{(2\pi i)^2} \oint_{\vert u'\vert=1} \oint_{\vert w'\vert=2} \frac{e^{\gamma^+u'^{-1}N/a+\gamma^-u'a/N}}{e^{\gamma^+w'^{-1}N/a+\gamma^-w'a/N}}\frac{u'^{x_i}(1-u'a/N)^{n_i}}{w'^{x_j+1}(1-w'a/N)^{n_j}}\frac{du'dw'}{w'-u'}\left(\frac{N}{a}\right)^{x_j-x_i}\\
=&\frac{1}{(2\pi i)^2} \oint_{\vert u'\vert=1} \oint_{\vert w'\vert=2} \frac{e^{u'^{-1}-at_iu'-w'^{-1}+at_jw'+O(1/N)}}{w'-u'}\frac{dudw}{u'^{-x_i}w'^{x_j+1}}(N/a)^{x_j-x_i}\\
\end{align*}
When taking the determinant, the term $(N/\sqrt{a})^{x_j-x_i}$ cancels. This gives the result.

\textbf{Remark}. Comparing this result to \cite{kn:BOO}, we see that the distribution of $\lambda^+$ converges to the Poissonized Plancherel measure for the symmetric groups. By the symmetry ($\lambda^{\pm}\leftrightarrow\lambda^{\mp},\gamma^{\pm}\leftrightarrow\gamma^{\mp}$) the same is true for $\lambda^-$. On the other hand, a similar contour integral argument to the above shows that $K(n_i,x_i;n_j,-n_j-x_j-1)\rightarrow 0$ as $N\rightarrow\infty$, which implies that $\lambda^+$ and $\lambda^-$ are asymptotically independent.
\end{proof}
\subsection{\texorpdfstring{Bulk Limits with $\gamma^{\pm}$ fixed}{Bulk Limits with gamma fixed}}
To state the next result, we need a definition. Given a complex number $z_+$ in the upper half plane, define

\[S_{z_+}(t_i-t_j;x_i-x_j)=\frac{1}{2\pi i}\int_{\overline{z_+}}^{z_+} u^{x_i-x_j-1}e^{(t_j-t_i)u}du.\]
If $t_i\geq t_j$, then the integration contour crosses $(0,\infty)$ but does not cross $(-\infty,0)$. If $t_i<t_j$, then the integration contour crosses $(-\infty,0)$ but not $(0,\infty)$. This kernel is one of the extensions of the discrete sine kernel constricted to \cite{kn:B}. A similar kernel appeared in \cite{kn:BO3}. It can be seen as a degeneration of the incomplete beta kernel, see Section \ref{incompletebeta}.

The main theorem of this section is the following:
\begin{theorem}\label{Asymptotics of Correlation Kernel}
Let $x_1,\ldots,x_k$ and $n_1,\ldots,n_k$ all depend on $N$ in such a way that $x_i-x_j$ is constant, $(n_j-N)/\sqrt{N}\rightarrow t_j\in\mathbb{R}$ and $x_i/\sqrt{N}\rightarrow c\in\mathbb{R}$ for all $1\leq i,j\leq k$.  Write $z_{+}$ for $(c+\sqrt{c^2-4\gamma^+})/2$. Then

\[\lim_{N\rightarrow\infty} \det[K(n_i,x_i;n_j,x_j)]_{1\leq i,j\leq k}=\]
\begin{align*}
\displaystyle && \left\{ \begin {array}{ll} 0, & \,\, c\geq 2\sqrt{\gamma^+}, \\
1, & \,\, c\leq -2\sqrt{\gamma^+},\\
\det[S_{z_+}(t_i-t_j;x_i-x_j)]_{1\leq i,j\leq k}, & \,\, -2\sqrt{\gamma^+}<c<2\sqrt{\gamma^+}.\\
\end{array} \right.
\end{align*}
\end{theorem}
\textbf{Remark}. Theorem~\ref{Asymptotics of Correlation Kernel} only makes a statement about the behavior around the top limit curve in Figure~\ref{Limit Curves}.  If we replace $x_i$ with $-x_i-n_i-1$ and $\gamma^+$ with $\gamma^-$, then by symmetry the same statement holds for the asymptotics around the lower Young diagram.
\begin{corollary}
Let $\rho_1(N,x)$ be the density function of $\cal P$$_N^{\gamma^+,\gamma^-}$. Then $\displaystyle\lim_{N\rightarrow\infty}\rho_1(N,\alpha N+\beta N^{1/2})$ equals

$\bullet$ 0, if $\alpha>0$ or $\alpha<-1$ or $\alpha=0,\beta\geq 2\sqrt{\gamma^+}$ or $\alpha=-1,\beta\leq -2\sqrt{\gamma^+}$.

$\bullet$ 1  if $-1\leq\alpha<0$ or $\alpha=0,\beta<-2\sqrt{\gamma^+}$ or $\alpha=-1,\beta\geq2\sqrt{\gamma^+}$.

$\bullet$ $\frac{1}{\pi}\arccos\left(\frac{\beta}{2\sqrt{\gamma^+}}\right)$ if $\alpha=0,-2\sqrt{\gamma^+}<\beta<2\sqrt{\gamma^+}$ or $\alpha=-1,-2\sqrt{\gamma^+}<\beta<2\sqrt{\gamma^+}$.
\end{corollary}
\begin{proof}
The arguments are similar to those used for the analysis of Plancherel measures for the symmetric groups in \cite{kn:OK}.

For reasons that will later become clear, it is more convenient to analyze $\sqrt{N}^{x_i-x_j}\sqrt{\gamma^+}^{x_j-x_i}K(n_i,x_i;n_j,x_j)$. When taking the determinant
\[\det[\sqrt{N}^{x_i-x_j}\sqrt{\gamma^+}^{x_j-x_i}K(n_i,x_i;n_j,x_j)],\]
the term $\sqrt{N}^{x_i-x_j}\sqrt{\gamma^+}^{x_j-x_i}$ cancels out.

We use the integral representation for the kernel in Theorem~\ref{sectiontwotheorem}. The conditions $n_i\geq n_j$ and $n_i<n_j$ translate to $t_i\geq t_j$ and $t_i<t_j$, respectively.

Just as in Theorem \ref{3.1}, we can deform the contours of integration as shown in Figure~\ref{Deformed}.

As $N\rightarrow\infty$, the integrand converges to $0$ for $\vert w\vert$ large enough because $\vert 1-w\vert\gg\vert 1-u\vert$. Therefore we can ignore the outer half of the $w$ contour. Then the contours of integration can be deformed to $\vert u\vert=1/\sqrt{N}$ and $\vert w\vert=2/\sqrt{N}$. Making the substitutions $u'=\sqrt{N}u$ and $w'=\sqrt{N}w$, the double integral is now

\begin{align*}
&\frac{1}{(2\pi i)^2} \oint_{\vert u\vert=1/\sqrt{N}} \oint_{\vert w\vert=2/\sqrt{N}} \frac{e^{\gamma^+u^{-1}+\gamma^-u}}{e^{\gamma^+w^{-1}+\gamma^-w}}\frac{u^{x_i}e^{n_i\ln(1-u)}}{w^{x_j+1}e^{n_j\ln(1-w)}}\frac{dudw}{w-u}\sqrt{N}^{x_i-x_j}\sqrt{\gamma^+}^{x_j-x_i}\\
=&\frac{1}{(2\pi i)^2} \oint_{\vert u'\vert=1} \oint_{\vert w'\vert=2} \frac{e^{\gamma^+u'^{-1}\sqrt{N}+\gamma^-u'/\sqrt{N}}}{e^{\gamma^+w'^{-1}\sqrt{N}+\gamma^-w'/\sqrt{N}}}\frac{u'^{x_i}e^{n_i\ln(1-u'/\sqrt{N})}}{w'^{x_j+1}e^{n_j\ln(1-w'/\sqrt{N})}}\frac{du'dw'}{w'-u'}\sqrt{\gamma^+}^{x_j-x_i}\\
=&\frac{1}{(2\pi i)^2} \oint_{\vert u'\vert=1} \oint_{\vert w'\vert=2} \frac{e^{\sqrt{N}(\gamma^+u'^{-1}+c\log u'-u'+O(1/\sqrt{N}))}}{e^{\sqrt{N}(\gamma^+w'^{-1}+c\log w'-w'+O(1/\sqrt{N}))}}\frac{du'dw'}{w'(w'-u')}\sqrt{\gamma^+}^{x_j-x_i}\\
\end{align*}

In general $\vert e^z\vert=e^{\Re z}$, so consider the real part of the function in the exponent,  $A(z)=\gamma^+z^{-1}+c\log z-z$. Note that $A'(z)=0$ at $z_{\pm}=\frac{c}{2}\pm\frac{\sqrt{c^2-4\gamma^+}}{2}$.

The basic idea of the rest of the proof can be summarized as follows. The term $\sqrt{\gamma^+}^{x_j-x_i}$ creates a $e^{\sqrt{N}(-c\log\sqrt{\gamma^+})}$ term in both the numerator and denominator. So it is equivalent to analyze $\Re(A(z)-c\log\sqrt{\gamma^+})=\Re(A(z)-A(z_+))$. We deform the $u$ and $w$ contours in such a way that $\Re(A(u)-A(z_+))<0$ and $\Re(A(w)-A(z_+))>0$, which will cause the integrand to converge to $0$ as $N\rightarrow 0$. However, the deformation of the contours causes the integral to pick up residues at $u=w$. These residues occur on a circular arc from $z_-$ to $z_+$. If $c=2\sqrt{\gamma^+}$, then $z_+=z_->0$, so the arc consists of a single point. As $c$ decreases, $z_+$ moves counterclockwise around the circle $\vert z\vert=\sqrt{\gamma^+}$ while $z_-$ moves clockwise. This means that the arc becomes increasingly large as $c$ decreases from $2\sqrt{\gamma^+}$ to $-2\sqrt{\gamma^+}$. When $c=-2\sqrt{\gamma^+}$, then $z_+=z_-<0$, so the arc has becomes the whole circle around the origin.

We then need to consider
\[-\frac{1}{2\pi i}\oint_{\vert z\vert=r<1}\frac{z^{x_i-x_j-1}}{(1-z)^{n_j-n_i}}dz,\]
which occurs when $n_i<n_j$. The expression for the residue at $u=w$ has the same integrand. With the minus sign, the integration contour for $z$ goes clockwise along a circle around the origin. Therefore it will cancel the circular arc from $z_-$ to $z_+$. This explains why the integration contour in $S_{z_+}$ crosses $(0,\infty)$ when $t_i\geq t_j$ and $(-\infty,0)$ when $t_i<t_j$.

\emph{Case 1:} $-2\sqrt{\gamma^+}<c<2\sqrt{\gamma^+}$. Observe that $\Re(A(z)-A(z_+))=0$ for all $\vert z\vert=\vert z_{\pm}\vert=\vert\sqrt{\gamma^+}\vert$. Also notice that $A(z)-A(z_+)$ has a double zero at $z_+$ and $z_-$. See Figure~\ref{RealPart}.

\begin{figure}[htp]
\centering
\caption{On the left is $\Re(A(z)-c\log{\sqrt{\gamma^+}})$, where the black regions indicate $\Re<0$ and the white regions indicate $\Re>0$.}
\includegraphics[totalheight=0.25\textheight]{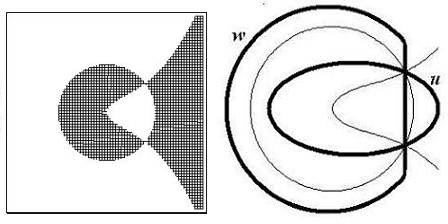}
\label{RealPart}
\end{figure}

If the contours of integration are deformed as shown in Figure~\ref{RealPart}, then

\[\frac{e^{\sqrt{N}(\gamma^+u'^{-1}+c\log u'-u'+O(1/\sqrt{N}))}}{e^{\sqrt{N}(\gamma^+w'^{-1}+c\log w'-w'+O(1/\sqrt{N}))}}\rightarrow 0\]
as $N\rightarrow\infty$. The integral thus approaches zero, except for the residues at $u=w$. So $\sqrt{N}^{x_i-x_j}K(n_i,x_i;n_j,x_j)$ converges to

\[\sqrt{N}^{x_i-x_j}\frac{1}{2\pi i}\int_{z_-/\sqrt{N}}^{z_+/\sqrt{N}}\frac{du}{u^{x_j-x_i+1}}(1-u)^{n_i-n_j}\rightarrow\frac{1}{2\pi i}\int_{z_-}^{z_+} u^{x_i-x_j-1}e^{-(t_i-t_j)u}du\].

If $t_i\geq t_j$, then the integration contour crosses $(0,\infty)$. If $t_i<t_j$, then the contour crosses $(-\infty,0)$.

\emph{Case 2:} $c^2-4\gamma^+>0$ and $c>0$. Deforming the contours of integration as shown in Figure~\ref{Realpart1}, the integral becomes zero. The contours do not pass through each other, so no residues appear. So $\sqrt{N}^{x_i-x_j}K(n_i,x_i;n_j,x_j)\rightarrow 0$ if $t_i\geq t_j$. This means that $\det[K(n_i,x_i;n_j,x_j)]\rightarrow 0$.

\begin{figure}[htp]
\centering
\caption{Again, the figure on the left shows $\Re(A(z)-A(z_+))$, with black regions indicating $\Re<0$ and white regions indicating $\Re>0$. }
\includegraphics{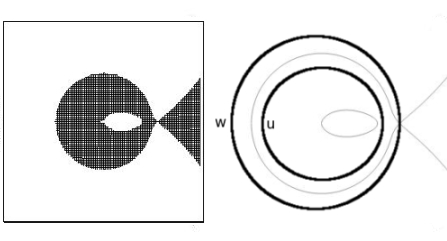}
\label{Realpart1}
\end{figure}

\emph{Case 3:} $c^2-4\gamma^+>0$ and $c<0$. Deform the contours as shown in Figure~\ref{Realpart3}. Since the $w$ and $u$ contours pass through each other during the deformation, the integral picks up residues at $u=w$. So if $t_i\geq t_j$, then $\sqrt{N}^{x_i-x_j}K(n_i,x_i;n_j,x_j)$ converges to

\begin{align*}
\sqrt{N}^{x_i-x_j}\frac{1}{2\pi i}\oint w^{x_i-x_j-1}(1-w)^{n_i-n_j}dw=&\frac{1}{2\pi i}\oint w^{x_i-x_j-1}e^{(t_j-t_i)w}dw\\
=&\frac{(t_j-t_i)^{x_j-x_i}}{(x_j-x_i)!}\\
\end{align*}

If $t_i<t_j$, then there is the integral in $z$, which cancels with the residues at $u=w$, so $\sqrt{N}^{x_i-x_j}K(n_i,x_i;n_j,x_j)$ converges to $0$. This means that the matrix $[K(n_i,x_i;n_j,x_j)]$ asymptotically has ones on the diagonal and zeroes below. So $\det[K(n_i,x_i;n_j,x_j)]$ converges to $1$.

\begin{figure}[htp]
\centering
\caption{On the left is $\Re(A(z)-c\log{\sqrt{\gamma^+}})$, where the black regions indicate $\Re<0$ and the white regions indicate $\Re>0$.}
\includegraphics{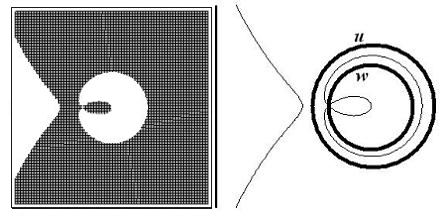}
\label{Realpart3}
\end{figure}
\end{proof}

\textbf{Remark.} It is natural to ask what happens when $x_i/\sqrt{n_i}$ do not all converge to the same real number. When this occurs, the determinant $\det[K(n_i,x_i;n_j,x_j)]$ factors into blocks corresponding to distinct values of $\lim x_i/\sqrt{n_i}$. Probabilistically, this means that the probability of finding a vertical edge becomes independent in different parts of the boundary.

\subsection{\texorpdfstring{Bulk Limits with $\gamma^{\pm}\propto N$}{Bulk Limits with gamma proportional to N}}\label{incompletebeta}
We now let $\gamma^{\pm}$ depend on $N$ in such a way that $\gamma^+/N\rightarrow a>0$ and $\gamma^-/N\rightarrow b>0$ as $N\rightarrow\infty$. Before we can state the result, some preliminary definitions and lemmas are needed.

For $a,b>0$ and $c\in\mathbb{R}$, recall that
\[R_{a,b,c}(z)=-bz^3+(b-c-1)z^2+(c+a)z-a.\]

\begin{lemma}\label{Rabc}
(1) The cubic polynomial $R_{a,b,c}(z)$ has a multiple root iff $c$ is a root of $Q_{a,b}(z)$, where $Q_{a,b}(z)$ is defined in \S 3.1.

(2) Let $q_1\leq\ldots\leq q_m$ be the real roots of $Q_{a,b}$. If $q_1<c<q_2$ or $q_{m-1}<c<q_m$ then $R_{a,b,c}(z)$ has a pair of complex conjugate roots.
\end{lemma}
\begin{proof}
(1) In general, a polynomial has a multiple root iff its discriminant is zero. The discriminant of $R_{a,b,c}$ is exactly $Q_{a,b}(c)/16$.

(2) A cubic polynomial has nonreal roots iff its discriminant is negative. Since $Q_{a,b}$ diverges to $+\infty$ in both directions, $Q_{a,b}(z)$ is negative for $q_1<z<q_2$ and $q_{m-1}<z<q_m$.
\end{proof}

\begin{lemma}\label{Qab}
The polynomial $Q_{a,b}$ has a double root at $c_0$ iff $a$, $b$ and $c_0$ satisfy the equations
\begin{align}\label{Aandb}
a=\frac{z_0^3}{(z_0-1)^3},\ \ b=-\frac{1}{(z_0-1)^3},\ \ c_0=-\frac{z_0^2(z_0-3)}{(z_0-1)^3}
\end{align}
for some $z_0\in\mathbb{R}$.
\end{lemma}
\begin{proof}
Since $Q_{a,b}(c_0)$ is the discriminant of $R_{a,b,c_0}$, $Q_{a,b}(z)$ has a double root at $c_0$ iff $R_{a,b,c_0}(z)$ has a triple root. For any $z_0$, $R$ has a triple root at $z_0$ iff $R(z_0)=R'(z_0)=R''(z_0)=0$. This gives three linear equations in the three variables $a$, $b$, and $c_0$, which can be solved explicitly.
\end{proof}
\textbf{Remark.} We have $a,b>0$ iff $z_0<0$. Then $-1<c_0<0$.
\begin{figure}[htp]
\caption{This figure shows the equations in (\ref{Aandb}), with $a$ plotted on the horizontal axis and $b$ plotted on the vertical with parameter $z_0$.}
\begin{center}
\includegraphics[scale=0.75]{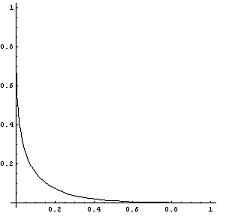}
\end{center}
\label{JustRightab}
\end{figure}

One more definition is needed before we can state the main result of
this section. Let $\mathrm{B}$ be the incomplete beta kernel defined
by
\[\mathrm{B}_z(k,l)=\frac{1}{2\pi i}\int_{\bar{z}}^z (1-u)^ku^{-l-1}du,\]
where the path of integration crosses $(0,1)$ for $k\geq 0$ and $(-\infty,0)$ for $k<0$. The incomplete beta kernel has been introduced in \cite{kn:OR2}. It is one of the extensions of the discrete sine kernel of \cite{kn:B}.

\begin{theorem}\label{4.1}
Let $\gamma^{+}/N\rightarrow a$ and $\gamma^{-}/N\rightarrow b$ for positive real numbers $a$ and $b$. Also let $x_1,\ldots,x_k$ and $n_1,\ldots,n_k$ depend on $N$ in such a way that $n_i-n_j$ and are $x_i-x_j$  constant, $n_j/N\rightarrow 1$ and $x_j/N\rightarrow c$ for all $1\leq i,j\leq k$.
Let $q_1\leq\ldots\leq q_m$ denote the distinct real roots of $Q_{a,b}(x)$ ($m$ can be 2,3, or 4).  Additionally, assume $Q_{a,b}(c)\neq 0$. Let $z_+$ be a root of $R_{a,b,c}(x)$ such that $\Im(z_+)\geq 0$ (cf. Lemma~\ref{Rabc}). If $m=4$, then
\begin{align*}
\det[K(n_i,x_i;n_j,x_j)]_{1\leq i,j\leq k} &\rightarrow & \left\{
\begin {array}{ll} 0, & \,\, c\leq q_1, \\
\det[\mathrm{B}(n_i-n_j,x_j-x_i;z_+)]_{1\leq i,j\leq k}, & \,\, q_1<c<q_2,\\
1, & \,\, q_2\leq c\leq q_3,\\
\det[\mathrm{B}(n_i-n_j,x_j-x_i;z_+)]_{1\leq i,j\leq k}, & \,\, q_3<c<q_4,\\
0, & \,\, c\geq q_4. \\
\end{array} \right.
\end{align*}
If $m=2$ or $3$, then
\begin{align*}
\det[K(n_i,x_i;n_j,x_j)]_{1\leq i,j\leq k} &\rightarrow & \left\{
\begin {array}{ll} 0, & \,\, c\leq q_1, \\
\det[\mathrm{B}(n_i-n_j,x_j-x_i;z_+)]_{1\leq i,j\leq k},& \,\, q_1<c<q_m,\\
0, & \,\, c\geq q_m, \\
\end{array} \right.
\end{align*}
\end{theorem}
\begin{proof}
The double integral in the correlation kernel of Theorem~\ref{sectiontwotheorem} asymptotically becomes
\[\left(\frac{1}{2\pi i}\right)^2\oint\oint\frac{e^{N(au^{-1}+bu+c\log(u)+\log(1-u)+O(1/N))}}{e^{N(aw^{-1}+bw+c\log(u)+\log(1-u)+O(1/N))}} \frac{dudw}{w(u-w)}\]
where the contours are over $\vert u\vert=r$ and $\vert w-1\vert=\epsilon<1-r$.
So we can perform a similar analysis as in Theorem~\ref{Asymptotics of Correlation Kernel}, except with a more complicated $A(z)=az^{-1}+bz+c\log(z)+\log(1-z)$. For this proof, it is actually more convenient to write $A(z;c)$ in place of $A(z)$.

\begin{figure}[htp]
\centering
\caption{The shaded regions show $\Re(A(z;c)-A(z_+;c))<0$, while the white regions show $\Re>0$. The first row corresponds to $c<q_1$, the second row corresponds to $q_1<c<q_2$, the third corresponds to $q_2<c<q_3$, the fourth corresponds to $q_3<c<q_4$, and the fifth corresponds to $c>q_4$.}
\fbox{\includegraphics[totalheight=0.30\textheight]{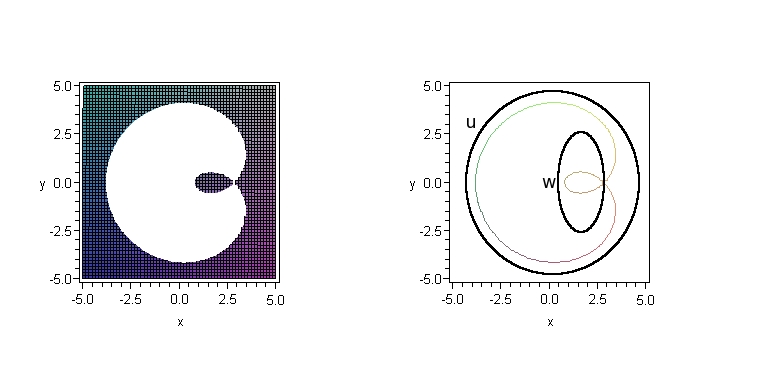}}

\vspace{0.1in}

\fbox{\includegraphics[totalheight=0.30\textheight]{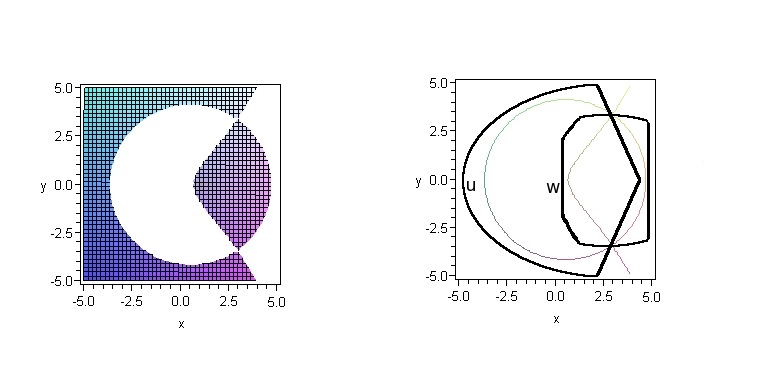}}

\label{Bulk}
\end{figure}
\begin{figure}[htp]
\centering
\fbox{\includegraphics[totalheight=0.30\textheight]{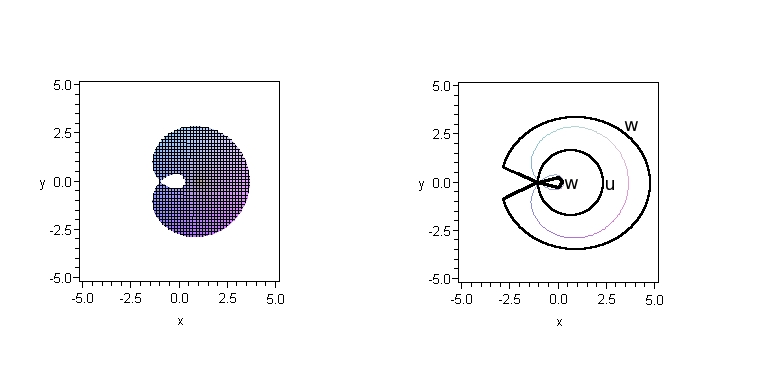}}

\vspace{0.1in}

\fbox{\includegraphics[totalheight=0.30\textheight]{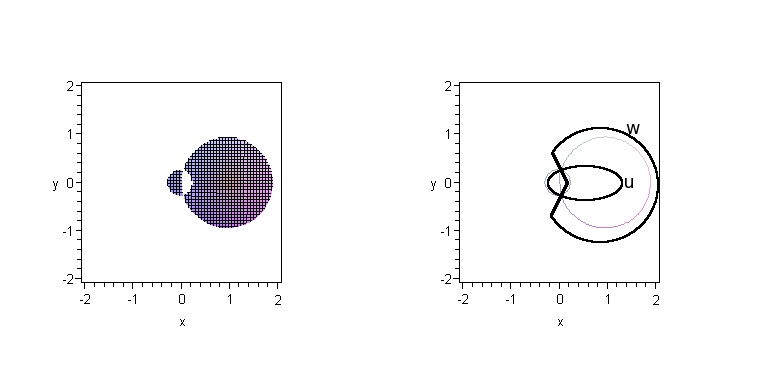}}

\vspace{0.1in}

\fbox{\includegraphics[totalheight=0.30\textheight]{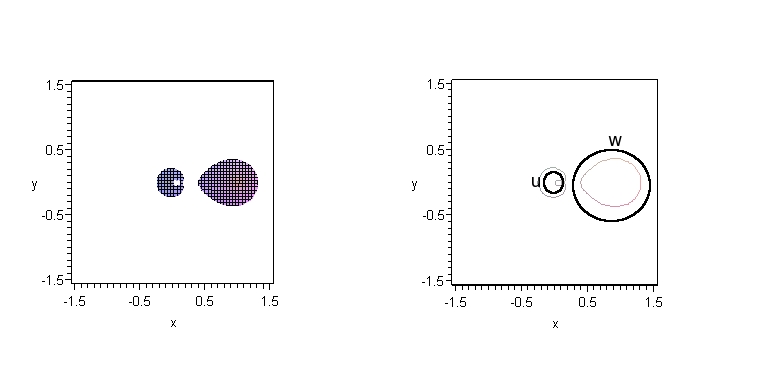}}

\end{figure}

First we find which values of $c$ correspond to the edges of the hypothetical limit shape in Figure~\ref{limitcurves2}. These are the values of $c$ such that $A(z;c)-A(z_0;c)$ has a triple zero for some $z_0\in\mathbb{C}$. Requiring $A(z;c)-A(z_0;c)$ to have a triple zero at $z=z_0$ is equivalent to requiring $A'(z;c)$ to have a double zero at $z=z_0$. Multiplying the equation $A'(z;c)=0$ by $z^2(1-z)$ gives the equation $R_{a,b,c}(z)=0$ (Note that $R_{a,b,c}(0)=-a$ and $R_{a,b,c}(1)=-1$, which are both nonzero). By Lemma~\ref{Rabc}, $R_{a,b,c}$ has a double zero iff $c=q_1,\ldots,q_m$.
%\[C_{a,b,c}(z)=-bz^3+(b-c-1)z^2+(c+a)z-a=0.\]
%Cubic equations can be solved explicitly, and setting two of the zeros equal reveals that $C_{a,b,c}(z)$ has a double zero when $c$ is a root of $Q_{a,b}(z)$.

%If $A(z;c_0)-A(z_0,c_0)$ has a quadruple root at $z=z_0$, then $C_{a,b,c_0}(z)$ has a triple root at $z=z_0$. Solving $-\frac{1}{b}C_{a,b,c_0}(z)=(z-z_0)^3$ for $a$ and $b$ gives the equations

Now we need to determine how to appropriately deform the contours.
The analysis here is almost identical to that of
Thereom~\ref{Asymptotics of Correlation Kernel}. We want to find
nonreal values of $z_0$ such that $A(z;c)-A(z_0;c)$ has a double
zero. This reduces to looking for nonreal roots of $R_{a,b,c}(z)$ in
the upper half-plane, which we have defined to be $z_+$. As can be
seen from Figure~\ref{limitcurves2}, there are potentially five
different regions of behavior for the bulk limits.  The
corresponding behavior of $\Re(A(z;c)-A(z_+;c))$ is shown in
Figure~\ref{Bulk}. (These are computer generated figures for
specific values of parameters, however, it is not hard to prove that
similar figures arise for any values of the parameters in the
corresponding domains. An example of such an argument can be found
in the beginning of the proof of Theorem \ref{Pearcey} below.) The
arguments of Theorem~\ref{Asymptotics of Correlation Kernel} are
again applicable here, except with the new definition of $z_{+}$.
\end{proof}

\subsection{The Pearcey Kernel as an Edge Limit}
We now find the edge limit at the point where the two limit curves in the middle figure in Figure~\ref{limitcurves2} just barely merge. In this case, we analyze the limiting behavior of $K_{\Delta}$ from Corollary~\ref{sectiontwocorollary} instead of $K$, which corresponds to the fact that we consider the limit of the point process formed by columns of $\lambda^{\pm}$ rather than by their rows, see Figure~\ref{Young configuration}.

\begin{theorem}\label{Pearcey}
Fix $z_0<0$ and let $a,b$ and $c_0$ satisfy equations (\ref{Aandb}). Let $\gamma^+/N\rightarrow a$ and $\gamma^-/N\rightarrow b$ as $N\rightarrow\infty$. Let $n_1,\ldots,n_k$ depend on $N$ in such a way that $(n_j-N)/\sqrt{N}\rightarrow 2t_j\in\mathbb{R}$ as $N\rightarrow\infty$. Set $\zeta=(z_0-1)\vert z_0\vert^{-1/2}<0$. Define
\[\tilde{t}_j=\frac{z_0}{1-z_0}t_j\]
and let $x_1,\ldots,x_k$ depend on $N$ in such a way that
\[\frac{\zeta(x_j-c_0N-\tilde{t}_j\sqrt{N})}{N^{1/4}}\rightarrow s_j\in\mathbb{R}\]
as $N\rightarrow\infty$.
Then as $N\rightarrow\infty$,
\[\det[-\zeta^{-1}N^{1/4}K_{\Delta}(n_i,x_i; n_j,x_j)]_{1\leq i,j\leq k}\rightarrow\det[P(t_i,s_i;t_j,s_j)]_{1\leq i,j\leq k}\]
where \begin{equation} \begin{gathered} P(t_i,s_i;t_j,s_j)\\=
\label{PearceyEquation}
\left(\frac{1}{2\pi i}\right)^2\int\int e^{w^4-u^4+t_iu^2-t_jw^2+s_iu-s_jw}\frac{dudw}{u-w}\\
\qquad\qquad\qquad\qquad\qquad-
\frac{1}{\sqrt{2\pi\vert t_i-t_j\vert}}\exp\left({-\frac{(s_j-s_i)^2}{2(t_i-t_j)}}\right), \qquad t_i> t_j\\
\qquad\qquad\qquad\left(\frac{1}{2\pi i}\right)^2\int\int
e^{w^4-u^4+t_iu^2-t_jw^2+s_iu-s_jw}\frac{dudw}{u-w}, \qquad t_i\leq
t_j
\end{gathered}
\end{equation}
where $u$ is integrated from $-i\infty$ to $i\infty$ and $w$ is
integrated on the rays from $\pm\infty e^{i\pi/4}$ to $0$ and from
$\pm\infty e^{-i\pi/4}$ to $0$ as in Figures \ref{PearceyContour1}
and \ref{PearceyContour2}.

\end{theorem}
\begin{figure}[htp]
\caption{The contour for $u$.}
\begin{center}
\includegraphics[scale=0.75]{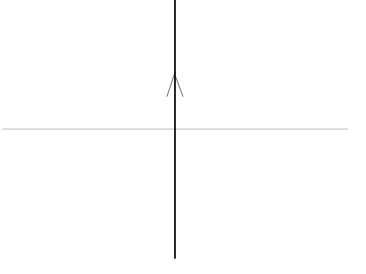}
\end{center}
\label{PearceyContour1}
\end{figure}
\begin{figure}[htp]
\caption{The contour for $w$.}
\begin{center}
\includegraphics[scale=0.75]{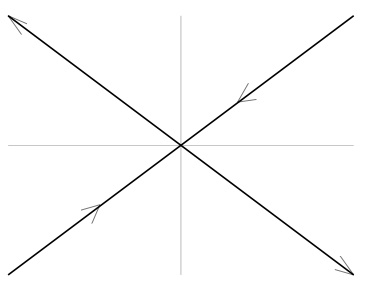}
\end{center}
\label{PearceyContour2}
\end{figure}
The kernel $P(t_i,s_i;t_j,s_j)$ is called the Pearcey kernel and it was previously obtained in \cite{kn:ABK},\cite{kn:BH},\cite{kn:BH2},\cite{kn:OR},\cite{kn:TW}.
\begin{proof}
The argument is similar to the proofs of Theorems~\ref{Asymptotics of Correlation Kernel} and \ref{4.1}.
It is convenient to let $A(z;c;d)$ denote $az^{-1}+bz+c\log z + d\log(1-z)$. Then the double integral in the correlation kernel of Corollary~\ref{sectiontwocorollary} becomes asymptotically
\begin{align}
&&-\left(\frac{1}{2\pi i}\right)^2&\int\int \frac{e^{N(au^{-1}+bu+(x_i/N)\log u+(n_i/N)\log(1-u)+O(1/N))}}{e^{N(aw^{-1}+bw+(x_j/N)\log w+(n_j/N)\log(1-w)+O(1/N))}} \frac{dudw}{w(u-w)}\\
&&=-\left(\frac{1}{2\pi i}\right)^2&\int\int \frac{e^{N(A(u;x_i/N;n_i/N)+O(1/N))}}{e^{N(A(w;x_j/N;n_j/N)+O(1/N))}} \frac{dudw}{w(u-w)}\label{PPP}
\end{align}

Multiplying the integrand by the conjugating factor

\[\frac{e^{-NA(z_0;x_i/N;n_i/N)}}{e^{-NA(z_0;x_j/N;n_j/N)}}=\frac{z_0^{-x_i}}{z_0^{-x_j}}\frac{(1-z_0)^{-n_i}}{(1-z_0)^{-n_j}}\frac{e^{-aNz_0^{-1}}}{e^{-aNz_0^{-1}}}\frac{e^{-bNz_0}}{e^{-bNz_0}},\]
which cancels when taking the determinant for correlation functions, allows us to consider $A(z;x_m/N;n_m/N)-A(z_0;x_m/N;n_m/N)$ instead of $A(z;x_m/N;n_m/N)$.

%Expanding into a Taylor series gives that
%\begin{align}
%A(z;c_0;1)-A(z_0;c_0;1)=p_4(z-z_0)^4+O((z-z_0)^5)
%\end{align}
%Substituting $z'=(-p_4N)^{1/4}(z-z_0)$, the above expression becomes
%\begin{align}
%N(A(z;c_0;1)-A(z_0;c_0;1))=-(z')^4+o(1)
%\end{align}
%More generally,
%\begin{align}
%N\left(A\left(z;c_0+\frac{c_{i3}}{N^{3/4}};1\right)-A\left(z_0;c_0+\frac{c_{i3}}{N^{3/4}};1\right)\right)=(-p_4)^{-1/4}c_{i3}\frac{z'}{z_0}-(z')^4+o(1)
%\end{align}

\begin{figure}[htp]
\centering
\caption{The figure on the left shows $\Re(A(z;c_0;1)-A(z_0;c_0;1))$, with black regions indiciating $\Re<0$ and white regions indicating $\Re>0$.}
\includegraphics[totalheight=0.25\textheight]{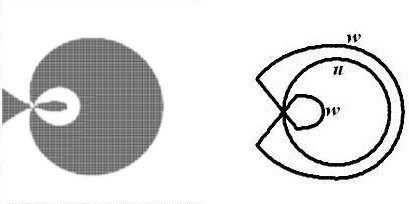}
\label{QuadrupleZero}
\end{figure}

Deform the contours as shown in Figure~\ref{QuadrupleZero}. Let us show that these contours exist. We know that the level lines only intersect at $z_0$ (the only critical point of the function $A(z;c_0;1)-A(z_0;c_0;1)$, since $A'(z)=-b(z-z_0)^3z^{-2}(1-z)^{-1}$), and they are symmetric with respect to the real axis. Restrict $\Re(A(z;c_0;1)-A(z_0;c_0;1))$ to the real axis. For $\vert x\vert=\epsilon$ small, the main contribution to $\Re(A(x;c_0;1))$ comes from the term $ax^{-1}$. So $\Re(A(x;c_0;1))$ is positive at $x=\epsilon>0$ and negative at $x=\epsilon<0$, so the level lines cross the real axis at $0$. For $x=1-\epsilon$ with $\epsilon$ small, the main contribution to $\Re(A)$ comes from the term $\log\vert 1-x\vert$. This implies that $\Re(A)$ is negative $x=1-\epsilon$, so the level lines cross the real axis somewhere between $0$ and $1$. For large $x$, the main contribution to $\Re(A)$ comes from $bx$, so $\Re(A)$ is positive for large $x$. Therefore the level lines cross the real axis at a third point. Since $A'(z)=-b(z-z_0)^3z^{-2}(1-z)^{-1}$ is positive for $z<z_0$, negative for $z\in(z_0,0)\cup(0,1)$, and positive for $z>1$, the levels lines can not intersect the real axis at any other point.

For a fixed $x\ll 0$, the main contribution to $\Re(A(x))$ comes from $bx$, so $\Re(A(x))$ is negative. However, as $y$ increases, $\Re(A(x+iy))$ goes to $+\infty$, since the main contributions come from $c_0\log\vert x+iy\vert+\log\vert 1-x-iy\vert$, and $c_0>-1$. This means there must be level lines going off to infinity. Restricting $\Re(A(z;c_0;1))$ to a circle $\vert z\vert=R\gg 1$ shows that these are the only level lines that go to infinity. Indeed, note that $\Re(A(z;c_0;1))>0$ if $z=R$, and as $z$ moves counterclockwise around the circle, the main contribution to the changes in $\Re(A(z))$ comes from $bz$. Thus $\Re(A(z))$ decreases as $z$ moves counterclockwise around the circle in the upper half-plane, so the circle can intersect at most one level line in the upper half-plane.

In the upper half-plane, there are four level lines coming from the critical point $z_0$. We know that three of these lines cross the real axis, while one of them goes off to infinity. Since they can only intersect at $z_0$, the only possibility is a picture as shown in Figure~\ref{QuadrupleZero}. This justifies the existence of the contours.

These deformations cause the kernel to pick up residues at $u=w$. The expression for these residues is
\begin{align}\label{sss}
-\frac{1}{2\pi i}\oint \frac{z^{x_i-x_j-1}}{(1-z)^{n_j-n_i}}dz
\end{align}
where the integral goes around a circle $\vert z\vert<1$. If $n_i\leq n_j$, then expression (\ref{sss}) cancels with the $z$-contour in expression (\ref{ttt}). If $n_i> n_j$, then explicitly evaluating the integral yields
\[-(-1)^{x_j-x_i}{n_i-n_j \choose x_j-x_i}.\]

The binomial can be approximated by the deMoivre-Laplace Theorem. For large $N$,
\[-N^{1/4}z_0^{x_j-x_i}(1-z_0)^{n_j-n_i}(-1)^{x_j-x_i}{n_i-n_j \choose x_j-x_i}\]
\[\approx -\frac{1}{\sqrt{2\pi(t_i-t_j)}}\exp\left({-\frac{(s_j-s_i)^2}{2(t_i-t_j)}}\right).\]
So when $t_i>t_j$, we obtain the extra exponential term in equation (\ref{PearceyEquation}).

For large values of $N$, all the contributions to the double integral come from near the point $z_0$. Taking the Taylor expansion around $z_0$ yields \[N\left(A\left(z;c_0+\frac{\tilde{t}_m}{N^{1/2}}+\frac{u_m}{N^{3/4}};1+\frac{2t_m}{N^{1/2}}\right)-A\left(z_0;c_0+\frac{\tilde{t}_m}{N^{1/2}}+\frac{u_m}{N^{3/4}};1+\frac{2t_m}{N^{1/2}}\right)\right)\]
\[=s_m z'+t_m(z')^2-(z')^4+o(1)\]
where $z'=z_0^{-1}\zeta^{-1}N^{1/4}(z-z_0)$.
This suggests the substitutions
\[u'=z_0^{-1}\zeta^{-1}N^{1/4}(u-z_0),\ \ w'=z_0^{-1}\zeta^{-1}N^{1/4}(w-z_0).\]
By making these substitutions, we are zooming in at the point $z_0$ in Figure~\ref{QuadrupleZero}. Then $u'$ is integrated as shown in Figure~\ref{PearceyContour1} while $w'$ is integrated as shown in Figure~\ref{PearceyContour2}.

The exponential terms in expression (\ref{PPP}) converge to the exponential terms in (\ref{PearceyEquation}). The term $\frac{dudw}{u-w}$ turns into $z_0\zeta N^{-1/4}\frac{du'dw'}{u'-w'}$. For large $N$, the contributions to the correlation kernel become focused around $z_0$, so the extra $w$ in the denominator becomes $z_0^{-1}$. The proof of Theorem~\ref{Pearcey} is complete.
\end{proof}

\subsection{The Airy Kernel as an Edge Limit}
Before stating the main result, some definitions are needed.

Let $Ai(x)$ denote the Airy function:
\[Ai(x)=\frac{1}{2\pi}\int_{-\infty}^{\infty}e^{is^3/3+ixs}ds.\]
This integral only converges conditionally. Shift the contour of integration as shown in Figure~\ref{AiryFunctionContour}. Along this contour, the function $e^{is^3/3}$ is real and decreases superexponentially.

\begin{figure}[htp]
\centering
\caption{A better contour for the Airy function. The contour goes from $\infty e^{5\pi i/6}$ to $0$ to $e^{\pi i/6}$.}
\includegraphics[totalheight=0.2\textheight]{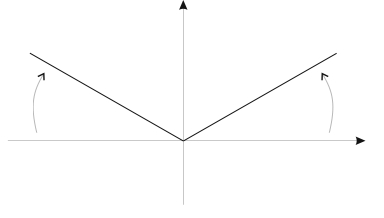}
\label{AiryFunctionContour}
\end{figure}

Define the $\textit{extended Airy kernel}$ $\cal A$ to be
\begin{align}
{\cal A}(\tau_1,\sigma_1;\tau_2,\sigma_2)=
\begin{cases}
\int_0^{\infty} e^{-\lambda(\tau_1-\tau_2)}Ai(\sigma_1+\lambda)Ai(\sigma_2+\lambda)d\lambda & \text{\ \ if } \tau_1\geq \tau_2,\\
-\int_{-\infty}^{0} e^{-\lambda(\tau_1-\tau_2)}Ai(\sigma_1+\lambda)Ai(\sigma_2+\lambda)d\lambda  & \text{\ \ if } \tau_1<\tau_2.
\end{cases}
\end{align}
It was first obtained in \cite{kn:PS} in the context of the polynuclear growth model.

There is a useful representation for ${\cal A}$ as a double integral.
\begin{proposition}\label{Johansson} (~\cite{kn:J}, \S 2.2) Let $\nu_1,\nu_2$ satisfy $\nu_1+\nu_2+\tau_1-\tau_2>0$. If $\tau_1\geq \tau_2$, then
\[{\cal A}(\tau_1,\sigma_1;\tau_2,\sigma_2)=\left(\frac{1}{2\pi i}\right)^2\int_{\Im(u)=\nu_1}\int_{\Im(w)=\nu_2}\frac{e^{i\sigma_1u+i\sigma_2w+i(w^3+u^3)/3}}{\tau_2-\tau_1+i(w+u)}dudw.\]
If $\tau_1<\tau_2$, then
\[{\cal A}(\tau_1,\sigma_1;\tau_2,\sigma_2)=\left(\frac{1}{2\pi i}\right)^2\int_{\Im(u)=\nu_1}\int_{\Im(w)=\nu_2}\frac{e^{i\sigma_1u+i\sigma_2w+i(w^3+u^3)/3}}{\tau_2-\tau_1+i(w+u)}dudw\]
\[-\frac{1}{\sqrt{4\pi(\tau_2-\tau_1)}}\exp\left(-\frac{(\sigma_1-\sigma_2)^2}{4(\tau_2-\tau_1)}-\frac{1}{2}(\tau_2-\tau_1)(\sigma_1+\sigma_2)+\frac{1}{12}(\tau_2-\tau_1)^3\right)\]
\end{proposition}

The double integral from Proposition~\ref{Johansson} can be rewritten as
\begin{multline}\label{AiryForm1}
\left(\frac{1}{2\pi i}\right)^2\int\int\exp\big(\tau_1\sigma_1-\tau_2\sigma_2-\frac{1}{3}\tau_1^3+\frac{1}{3}\tau_2^3-(\sigma_1-\tau_1^2)u+(\sigma_2-\tau_2^2)w\\
-\tau_1u^2+\tau_2w^2+\frac{1}{3}(u^3-w^3)\big)\frac{dudw}{u-w}.
\end{multline}
Indeed, just as we deformed the contours of integration for $Ai(x)$, we can deform the contours of integration in Proposition~\ref{Johansson}. The $u$-contour can be taken over $i\nu_1+\infty e^{5\pi i/6}$ to $i\nu_1$ to $i\nu_1+e^{\pi i/6}$, while the $w$-contour can be taken from $i\nu_2+\infty e^{5\pi i/6}$ to $i\nu_2$ to $i\nu_2+e^{\pi i/6}$. Integrating along these contours also allows for the possibility of $\nu_1+\nu_2+\tau_1-\tau_2=0$.
If we further make the substitutions $w=-iw'+\nu_2i$ and $u=iu'+\nu_1i$, then the double integral becomes
\begin{align*}
&\left(\frac{1}{2\pi i}\right)^2\int\int\exp\big(-\nu_1\sigma_1-\nu_2\sigma_2+\frac{1}{3}\nu_1^3+\frac{1}{3}\nu_2^3-(\sigma_1-\nu_1^2)u+(\sigma_2-\nu_2^2)w\\
&+\nu_1u^2+\nu_2w^2+\frac{1}{3}(u^3-w^3)\big)\frac{dudw}{-\tau_2+\tau_1+\nu_1+\nu_2+u-w}.
\end{align*}
where $u$ is integrated from $\infty e^{-\pi i/3}$ to $0$ to $\infty e^{\pi i/3}$ and $w$ is integrated from $\infty e^{4\pi i/3}$ to $0$ to $\infty e^{2\pi i/3}$.
Taking $\nu_1=-\tau_1$ and $\nu_2=\tau_2$ turns the double integral into (\ref{AiryForm1}).
Writing the double integral in this form is useful when proving the following result.

In the next statement, let $Q_{a,b}$ be the same polynomial as in \S 3.1, see also \S 3.4.
\begin{theorem}\label{Airy}
Let $\gamma^+/N\rightarrow a,\gamma^-/N\rightarrow b$ for positive real numbers $a$ and $b$. Let $c_1$ be a root of $Q_{a,b}(z)$ and $z_1$ be the double zero of $R_{a,b,c_1}(z)$. Let $n_1,\ldots,n_k$ depend on $N$ in such a way that
\[\frac{n_j-N}{N^{2/3}}\rightarrow t_j\in\mathbb{R}\ \mathrm{as}\ N\rightarrow\infty.\]
Let $\tilde{t}_j=t_jz_1(1-z_1)^{-1}$ and let $x_1,\ldots,x_k$ depend on $N$ in such a way that
\[\frac{x_j-c_1N-\tilde{t}_jN^{2/3}}{N^{1/3}}\rightarrow s_j\in\mathbb{R}\ \mathrm{as}\ N\rightarrow\infty.\]
If $c_1>0$ or $c_1<-1$, set ${\cal K}=K$. Otherwise, set ${\cal K}=K_{\Delta}$. Then as $N\rightarrow\infty$,
\[\det[\vert z_1p_3^{1/3}\vert N^{1/3}{\cal K}(n_i,x_i; n_j,x_j)]_{1\leq i,j\leq k}\rightarrow\det[{\cal A}(\tau_i,\sigma_i;\tau_j,\sigma_j)]_{1\leq i,j\leq k}.\]

Here, $p_3$ denotes the constant
\[-\frac{1}{(1-z_1)^3}-\frac{3a}{z_1^4}+\frac{c_1}{z_1^3}\]
and
\[\tau_{m}=\frac{t_m}{2(p_3)^{2/3}(z_1-1)^2z_1},\ \ \sigma_{m}=\tau_m^2-\frac{s_m}{z_1p_3^{1/3}},\ \ 1\leq m\leq k.\]
\end{theorem}
\textbf{Remark}. The statement may seem a bit cryptic. Let us explain it in words. There are (potentially) four edge points as seen in Figure~\ref{limitcurves2}. We consider $K$ for the first point (when $c_1>0$) and the fourth point (when $c_1<-1$), which means that we look at the largest rows of $\lambda^+$ and $\lambda^-$. For the second and third points we consider $K_{\Delta}$, which means that we look at the largest columns of $\lambda^+$ and $\lambda^-$. For the second and fourth points, $\det[z_1p_3^{1/3}{\cal K}]\rightarrow\det[{\cal A}]$, while for the first and third points $\det[-z_1p_3^{1/3}{\cal K}]\rightarrow\det[{\cal A}]$. At the second and fourth points $z_1p_3^{1/3}$ is positive, while at the first and third points $z_1p_3^{1/3}$ is negative. This corresponds to the fact that in order to obtain the Airy process we need to flip the sign of particles at the lower edges of $\lambda^+$ and $\lambda^-$ (the second and fourth edge points, respectively).

\begin{proof}
This proof is similar to the proof of Theorem~\ref{Pearcey}, so some of the details will be omitted.

Once again, let $A(z;c;d)$ denote $az^{-1}+bz+c\log z+d\log(1-z)$. Multiplying by the conjugating factor
\[\frac{e^{-NA(z_1;x_i/N;n_i/N)}}{e^{-NA(z_1;x_j/N;n_j/N)}}=\frac{z_1^{-x_i}}{z_1^{-x_j}}\frac{(1-z_1)^{-n_i}}{(1-z_1)^{-n_j}}\frac{e^{-aNz_1^{-1}}}{e^{-aNz_1^{-1}}}\frac{e^{-bNz_1}}{e^{-bNz_1}}\]
allows us to consider $A(z;x_m/N;n_m/N)-A(z_1;x_m/N;n_m/N)$ instead of $A(z;x_m/N;n_m/N)$.
The Taylor expansion yields
\begin{align*}
&N\left(A\left(z;c_1+\frac{\tilde{t}_m}{N^{1/3}}+\frac{u_m}{N^{2/3}};1+\frac{t_m}{N^{1/3}}\right)-A\left(z_1;c_1+\frac{\tilde{t}_m}{N^{1/3}}+\frac{u_m}{N^{2/3}};1+\frac{t_m}{N^{1/3}}\right)\right)\\
&=\frac{1}{3}(z')^3-\frac{t_m}{2(p_3)^{2/3}(z_1-1)^2z_1}(z')^2+\frac{s_m}{(p_3)^{1/3}z_1}z'+o(1)
\end{align*}
where $z'=(p_3)^{1/3}N^{1/3}(z-z_1)$.
The contours of integration for $u$ and $w$ are shown in Figure~\ref{AIRYBIG}. Now let $u'=(p_3)^{1/3}N^{1/3}(u-z_1)$ and $w'=(p_3)^{1/3}N^{1/3}(w-z_1)$. Just like in the proof of Theorem~\ref{Pearcey}, the Taylor series gives rise to the exponential terms in~\ref{AiryForm1}. In addition, the term $\frac{dudw}{u-w}$ becomes $N^{-1/3}p_3^{-1/3}$, while the extra $w$ in the denominator becomes $z_1^{-1}$. We break down the following analysis into cases.

\begin{figure}[htp]
\centering
\caption{The left column shows $\Re(A(z;c_0;1)-A(z_0;c_0;1))$, with shaded regions showing $\Re<0$ and white regions showing $\Re>0$. The right column shows the local behavior around $z_1$. The first row occurs when $c_1$ is the smallest real root of $Q_{a,b}$, the second row when $c_1$ is the second smallest real root, and so forth. If $Q_{a,b}$ has only two real roots, the middle two rows do not occur.}
\includegraphics[totalheight=0.83\textheight]{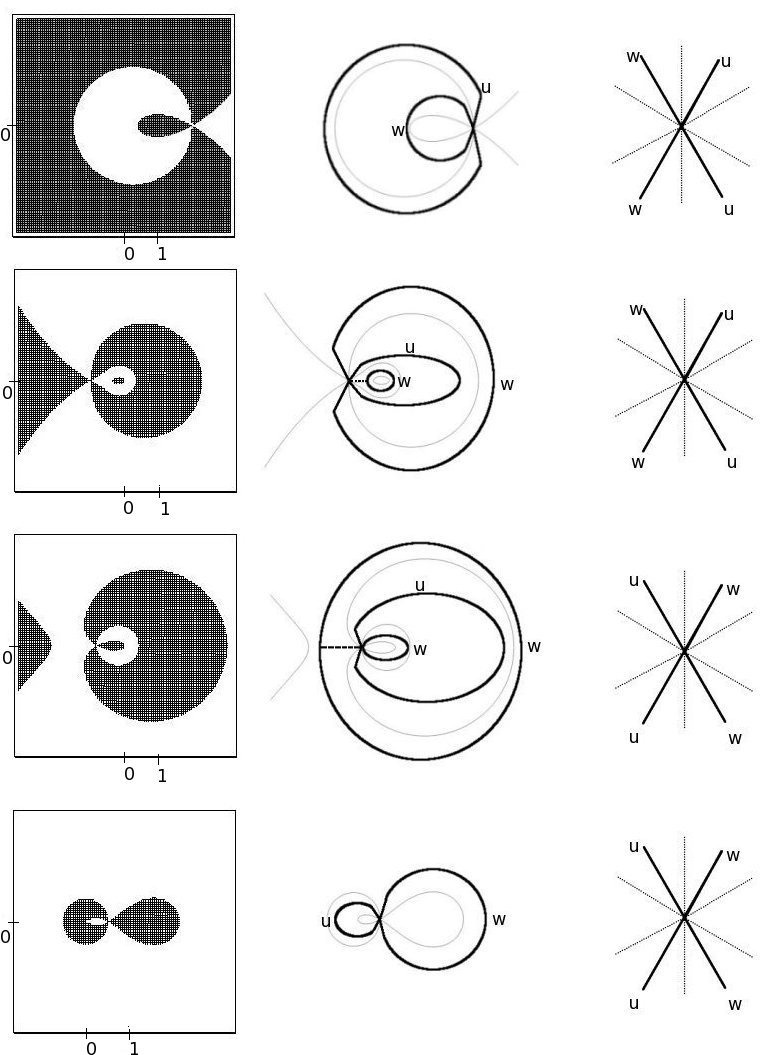}
\label{AIRYBIG}
\end{figure}

\textit{Case 1:} $c_1>0$. This corresponds to the fourth row in Figure~\ref{AIRYBIG} and the top edge point of Figure~\ref{Limit Curves}. In this case, $p_3$ is negative, so the contours for $u'$ and $w'$ agree with the contours in expression (\ref{AiryForm1}). Since $0<z_1<1$, this implies that $\tilde{t}_j-\tilde{t}_i>0$ if $t_j-t_i>0$. Since $n_j>n_i$ translates to $t_j>t_i$, this means that $x_j-x_i$ can be assumed positive if $n_j>n_i$. Therefore the integral in $z$ from expression (\ref{uuu}) can be written as
\[-{n_j-n_j+x_j-x_i-1 \choose x_j-x_i}=-{n_j-n_i+x_j-x_i \choose x_j-x_i}\frac{n_j-n_i}{n_j-n_i+x_j-x_i}.\]
Using the Laplace-Demoivre Theorem shows that
\[-N^{1/3}\frac{z_1^{x_j-x_1}}{(1-z_1)^{n_i-n_j}}{n_j-n_i+x_j-x_i-1 \choose x_j-x_i}\rightarrow\]
\begin{align}\label{extraexpterm}
-\frac{\vert 1-z_1\vert}{\sqrt{2\pi \vert z_1\vert(t_j-t_i)}}\exp\left(-\frac{(1-z_1)^2}{2\vert z_1\vert}\frac{(s_j-s_i)^2}{\vert t_j-t_i\vert}\right).
\end{align}
Taking the last term in Proposition~\ref{Johansson} and multiplying by $\exp(-\tau_1\sigma_1+\tau_2\sigma_2+\frac{1}{3}\tau_1^3-\frac{1}{3}\tau_2^3)$ yields
\[-\vert p_3\vert^{1/3}\vert z_1\vert^{1/2}\frac{\vert 1-z_1\vert}{\sqrt{2\pi(t_j-t_i)}}\exp\left(-\frac{(1-z_1)^2}{2z_1}\frac{(s_j-s_i)^2}{t_j-t_i}\right).\]
We have seen that
\begin{multline}\label{yay}
\vert z_1p_3^{1/3}\vert N^{1/3}\frac{z_1^{x_j-x_1}}{(1-z_1)^{n_i-n_j}}K(n_i,x_i;n_j,x_j)\rightarrow\\
\exp(-\tau_1\sigma_1+\tau_2\sigma_2+\frac{1}{3}\tau_1^3-\frac{1}{3}\tau_2^3){\cal A}(\tau_i,\sigma_i;\tau_j,\sigma_j),
\end{multline}
which gives the result.

\textit{Case 2:} $c_1<-1$. This corresponds to the first row in Figure~\ref{AIRYBIG}. Here, $z_1>1$ and $p_3>0$. Making the deformations gives residues at $u=w$, which can be written as
\[-\frac{1}{2\pi i}\oint_{\vert z-1\vert=\epsilon<1} \frac{z^{x_i-x_j-1}}{(1-z)^{n_j-n_i}}dz.\]
If $n_i\geq n_j$, then these residues are zero. If $n_i<n_j$, then $t_i<t_j$, which implies $x_i>x_j$, so the integral in $z$ from expression (\ref{uuu}) is zero. So when $n_i<n_j$, the extra term can be written as
\[(-1)^{n_j-n_i-1}{x_i-x_j-1 \choose n_j-n_i-1}=(-1)^{n_j-n_i-1}{x_i-x_j \choose n_j-n_i}\frac{n_j-n_i}{x_i-x_j}.\]
Using Laplace-Demoivre, this binomial converges to~\ref{extraexpterm}. So expression (\ref{yay}) holds.

\textit{Case 3:} $-1<c_1<0$. If $a$ and $b$ are small enough, then $Q_{a,b}$ has two roots between $-1$ and $0$. The second row in Figure~\ref{AIRYBIG} corresponds to the smaller root, while the third row corresponds to the larger root. In the second row $p_3$ is positive, while in the third row $p_3$ is negative. In both rows $z_1<0$.

Making the deformations gives residues at $u=w$, which can be written as
\[-\frac{1}{2\pi i}\oint_{\vert z\vert=r<1} \frac{z^{x_i-x_j-1}}{(1-z)^{n_j-n_i}}dz.\]
If $n_i\leq n_j$, then this expression cancels with the $z$-integral in expression (\ref{ttt}). If $n_i>n_j$, then the extra term can be written as
\[-(-1)^{x_j-x_i}{n_i-n_j \choose x_j-x_i}.\]
Once again, this converges to expression (\ref{extraexpterm}). So expression (\ref{yay}) holds.
\end{proof}

\end{document}